\documentclass[11pt,reqno]{amsart}

\usepackage{amsmath,amsthm,amssymb,amscd,MnSymbol,enumerate,paralist,hyperref,color}

\hypersetup{colorlinks=false}

\definecolor{darkgreen}{cmyk}{1,0,1,.2}
\definecolor{m}{rgb}{1,0.1,1}

\newdimen\theight
\def\TeXref#1{%
             \leavevmode\vadjust{\setbox0=\hbox{{\tt
                     \quad\quad  {\small \textrm #1}}}%
             \theight=\ht0
             \advance\theight by \lineskip
             \kern -\theight \vbox to
             \theight{\rightline{\rlap{\box0}}%
             \vss}%
             }}%

\theoremstyle{plain}

\newtheorem{thm}{Theorem}[section]
\newtheorem{lem}[thm]{Lemma}
\newtheorem{cor}[thm]{Corollary}
\newtheorem{prop}[thm]{Proposition}
\newtheorem{mainthm}{Theorem}

\theoremstyle{definition}

\theoremstyle{remark}

\newtheorem{claim}{Claim}

\newtheorem*{ack}{Acknowledgment}

\newcommand{\N}{\mathbb{N}}
\newcommand{\Z}{\mathbb{Z}}
\newcommand{\R}{\mathbb{R}}

\newcommand{\CC}{\mathcal{C}}

\newcommand{\EE}{\mathcal{E}}
\newcommand{\FF}{\mathcal{F}}
\newcommand{\GG}{\mathcal{G}}
\newcommand{\HH}{\mathcal{H}}
\newcommand{\II}{\mathcal{I}}
\newcommand{\JJ}{\mathcal{J}}
\newcommand{\KK}{\mathcal{K}}

\newcommand{\MM}{\mathcal{M}}
\newcommand{\NN}{\mathcal{N}}
\newcommand{\OO}{\mathcal{O}}
\newcommand{\PP}{\mathcal{P}}

\newcommand{\UU}{\mathcal{U}}
\newcommand{\VV}{\mathcal{V}}

\newcommand{\fH}{\mathfrak{H}}

\newcommand{\bm}{\mathbf{m}}
\newcommand{\bv}{\mathbf{v}}

\newcommand{\supp}{\operatorname{supp}}

\newcommand{\im}{\operatorname{im}}

\newcommand{\id}{\operatorname{id}}

\newcommand{\Diffeo}{\operatorname{Diffeo}}

\newcommand{\Hol}{\operatorname{Hol}}

\newcommand{\Prop}{\operatorname{Prop}}
\newcommand{\Emb}{\operatorname{Emb}}

\newcommand{\SO}{\operatorname{SO}}

\newcommand{\ol}{\overline}
\newcommand{\germ}{\boldsymbol{\gamma}}

\DeclareMathOperator{\Homeo}{Homeo}
\DeclareMathOperator{\dom}{dom}

\DeclareMathOperator{\hol}{hol}
\DeclareMathOperator{\Hess}{Hess}

\DeclareMathOperator{\grad}{grad}

\newcommand{\const}{\operatorname{const}}

\title{Molino's description and foliated homogeneity}

\author[J.A. \'Alvarez L\'opez]{Jes\'us A. \'Alvarez L\'opez}
\email{jesus.alvarez@usc.es}

\author[R. Barral Lij\'o]{Ram\'on Barral Lij\'o}
\email{ramonbarrallijo@gmail.com}

\address{Department of Mathematics,
         Institute of Mathematics\\
         University of Santiago de Compostela\\
         15782 Santiago de Compostela\\ Spain}
\thanks{The authors are partially supported by the grants FEDER/Ministerio de Ciencia, Innovaci\'on y Universidades/AEI/MTM2017-89686-P and MTM2014-56950-P, and Xunta de Galicia/2015 GPC GI-1574. The second author is also supported by a Canon Foundation in Europe Research Grant.}

\date{\today}



\begin{document}

\maketitle

\begin{abstract}
	The first author and Moreira Galicia have studied a topological version of Molino's theory. It describes equicontinuous foliated spaces satisfying certain conditions of strong quasi-analyticity, reducing their study to the particular case of $G$-foliated spaces. That description is sharpened in this paper by introducing a foliated action of a compact topological group on the resulting $G$-foliated space, like in the case of Riemannian foliations. A $C^\infty$ version is also studied. The triviality of this compact group characterizes compact minimal $G$-foliated spaces, which are also characterized by their foliated homogeneity in the $C^\infty$ case. We also give an example where the projection of the Molino's description is not a principal bundle, and another example of positive topological codimension where the foliated homogeneity cannot be checked by only comparing pairs of leaves---in the case of zero topological codimension, weak solenoids with this property were given by Fokkink and Oversteegen, and later by Dyer, Hurder and Lukina.
\end{abstract}

\tableofcontents

\addtocontents{toc}{\protect\setcounter{tocdepth}{1}}

\section{Introduction}\label{s: intro}

In Molino's theory \cite{Molino1982,Molino1988}, every minimal Riemannian foliation on a closed manifold can be described as a quotient of a minimal Lie foliation on another closed manifold by the foliated action of a compact Lie group, defining a principal bundle. Among minimal Riemannian foliations on closed manifolds, the Lie foliations are just the transitive ones. Here, transitivity means that the evaluation of the infinitesimal transformations of the foliation at every point is the whole tangent space. This property is equivalent to foliated homogeneity; i.e., the transitivity of the canonical left action of the group of foliated diffeomorphisms (diffeomorphisms that map leaves to leaves).

Ghys \cite[Appendix E]{Molino1988} suggested that equicontinuous foliated spaces should be considered as the topological Riemannian foliations. As a confirmation of this interpretation, the first author and Moreira Galicia \cite{AlvarezMoreira2016} gave a topological version of Molino's theory for any compact minimal equicontinuous foliated space $X$ whose holonomy pseudogroup and its closure are strongly quasi-analytic (see Section~\ref{ss: pseudogrs}). It describes $X$ as a foliated quotient of a $G$-foliated space $\widehat X_0$ for some locally compact local group $G$. Here, $G$-foliated spaces are the foliated spaces whose holonomy pseudogroup can be represented by a pseudogroup on some locally compact local group $G$ generated by some local left translations (a precise definition is given in Section~\ref{ss: foliated sps}). In the case where $G$ is a Lie group, the $G$-foliated spaces are the Lie foliations, which agrees with the original Molino's theory. According to the role played by Molino's theory in the study of Riemannian foliations, its topological version should have interesting applications; for instance, it was already used in \cite{AlvarezMoreira2016} to study the growth of leaves. Some features of the original Molino's description were missing in that topological version, like the foliated action of a compact Lie group on $\widehat X_0$, defining a principal bundle over $X$, and the characterization of $G$-foliated spaces using foliated homogeneity. For foliated spaces, foliated homogeneity means the transitivity of the canonical left action of the group of foliated homeomorphisms.

An interesting special case of foliated spaces is given by matchbox manifolds, which are the compact connected minimal foliated spaces of topological codimension zero; i.e., with totally disconnected local transversals. Dyer, Hurder and Lukina \cite{DyerHurderLukina2017} also gave an analogue of Molino's theory for equicontinuous matchbox manifolds, describing them as quotients of homogeneous matchbox manifolds by the action of a compact topological group, called the discriminant, introduced by the same authors in \cite{DyerHurderLukina2016}. An advantage of their construction is that it works without any additional condition, but their description is unique just when the hypotheses about strong quasi-analyticity used in \cite{AlvarezMoreira2016} are fulfilled. Their description is based on the work of Clark and Hurder \cite{ClarkHurder2013}, showing that a matchbox manifold is equicontinuous if and only if it is a weak solenoid (an inverse limit of a tower of covering maps between closed connected manifolds), and it is homogeneous if and only if it is a McCord solenoid (the covering maps can be chosen to be regular); sometimes the term strong solenoid is used for McCord solenoids. Since McCord solenoids are transversely modeled by left translations on profinite groups, they are special cases of $G$-foliated spaces. Thus the Molino's description of Dyer, Hurder and Lukina is a procedure to construct McCord solenoids from weak solenoids. This homogeneity characterization of McCord solenoids agrees with the foliated homogeneity characterization of minimal Lie foliations because  every homeomorphism between matchbox manifolds is foliated since the leaves are their path connected components. It turns out that the triviality of the discriminant characterizes homogeneous matchbox manifolds.

In the present paper, we show that some of these additional features of Molino's theory for equicontinuous matchbox manifolds and Riemannian foliations can be extended to equicontinuous foliated spaces with the indicated strong quasi-analyticity condition. Namely, such a foliated space is also described as quotient of some $G$-foliated space by the foliated action of certain discriminant group, and minimal compact $G$-foliated spaces are characterized using foliated homogeneity. Thus the triviality of the discriminant also characterizes homogeneous minimal compact foliated spaces.  

Let us state our main results and give more detailed observations. The terminology and notation used here are recalled in Section~\ref{s: prelim}. Our first goal is to show the following slight sharpening of the main result of the topological Molino's theory (Section~\ref{s: Molino}).

\begin{mainthm}[{Cf.\ \cite[Theorem~A]{AlvarezMoreira2016}}]\label{mt: Molino}
	Let $X\equiv(X,\FF)$ be a compact, minimal and equicontinuous foliated space, whose holonomy pseudogroup and its closure are strongly quasi-analytic. Then there is a locally compact local group $G$, a compact topological group $H$, a compact minimal $G$-foliated space $\widehat X_0\equiv(\widehat X_0,\widehat\FF_0)$, a foliated map $\hat\pi_0:\widehat X_0\to X$, and a free foliated right $H$-action on $\widehat X_0$ such that the restrictions of $\hat\pi_0$ to the leaves of $\widehat X_0$ are the holonomy coverings of the leaves of $X$, and $\hat\pi_0$ induces a homeomorphism $\widehat X_0/H\to X$. 
\end{mainthm}

Our new contribution in Theorem~\ref{mt: Molino} is the existence of $H$ satisfying the stated properties. Let us recall some ideas of the construction of $\widehat X_0$ from  \cite{AlvarezMoreira2016} to explain the definition and properties of $H$. Let $\HH$ be the representative of the holonomy pseudogroup of $X$ on a space $T=\bigsqcup_iT_i$ induced by the choice of a good foliated atlas $\{U_i\equiv B_i\times T_i\}$. Its closure is denoted by $\overline\HH$. Fix some $u_0\in T$, and let $\widehat T_0$ be the space of the germs of maps in $\overline\HH$ with source $u_0$, equipped with a locally compact Polish topology, which is different from the sheaf topology. A pseudogroup $\widehat\HH_0$ on $\widehat T_0$ is generated by the maps $\hat h:\hat\pi_0^{-1}(\dom h)\to\hat\pi_0^{-1}(\im h)$ ($h\in\HH$), defined by taking the left product with the germs of $h$. A locally equivariant projection $\hat\pi_0:\widehat T_0\to T_0$ is defined by the target map on germs. We have $\widehat T_0=\bigsqcup_i\widehat T_{i,0}$, where $\widehat T_{i,0}=\hat\pi_0^{-1}(T_i)$. Then, roughly speaking, $\widehat X_0$ is locally defined by replacing $T_i$ with $\widehat T_{i,0}$ in every $U_i$, obtaining a foliated structure defined by a foliated atlas $\{\widehat U_{i,0}\equiv B_i\times\widehat T_{i,0}\}$, and obtaining a foliated map $\hat\pi_0:\widehat X_0\to X$ induced by $\hat\pi_0:\widehat T_0\to T$. From the hypothesis on $\HH$ and $\overline\HH$, it follows that there is some locally compact local group $G$ and some finitely generated dense sublocal group $\Gamma\subset G$ such that $\widehat\HH_0$ is equivalent to the pseudogroup $\GG$ on $G$ generated by the local left translations by elements of $\Gamma$. 

Now let $H=\hat\pi_0^{-1}(u_0)$, which is a compact topological group with the germ product operation. The germ product also defines a right action of $H$ on $\widehat T_0$ such that $\hat\pi_0:\widehat T_0\to T$ induces an identity $\widehat T_0/H\equiv T$.  Using the definition of $\widehat X_0$, we get an induced foliated right action of $H$ on $\widehat X_0$ so that $\hat\pi_0:\widehat X_0\to X$ induces an identity $\widehat X_0/H\equiv X$. Via a pseudogroup equivalence $\widehat\HH_0\sim\GG$, $H$ can be considered as a compact subgroup of $G$, $\widehat\pi_0$ corresponds to the canonical projection $G\to G/H$, and $\HH$ becomes equivalent to the pseudogroup $\GG/H$ on $G/H$ generated by the induced local left action of $\Gamma$.

Under the hypothesis of Theorem~\ref{mt: Molino}, we also prove the following additional properties:
	\begin{enumerate}[(a)]
	
		\item\label{i: (G, Gamma, H, widehat X_0, hat pi_0)} The construction of $(G,\Gamma,H,\widehat X_0,\hat\pi_0)$ is independent of the choices involved up to an obvious equivalence relation (Proposition~\ref{p: equivalent Molino}). 
		
		\item $X$ is a $G$-foliated space for some local group $G$ if and only if $H$ is trivial (Proposition~\ref{p: G-foliated space <=> H = e}). 
		
		\item There is a subgroup in $H$ isomorphic to the holonomy group of every leaf (Proposition~\ref{p: Hol(L_x_0,x_0) subset H}). 
		
		\item $H$ contains no non-trivial normal subgroup of $G$ (Proposition~\ref{p: H has no non-trivial normal subgroups of G}).
		
		\item\label{i: C^infty Molino's description} If $X$ is $C^\infty$, then its Molino's description becomes $C^\infty$ in a unique obvious sense (Proposition~\ref{p: C^infty Molino}).
		
		\item The map $\hat\pi_0$ may not be a fiber bundle (an example is given in Section~\ref{ss: hat pi_0 is not a principal bundle}). This is the only missing property when comparing with the cases of Riemannian foliations and equicontinuous matchbox manifolds.
		
	\end{enumerate}
According to~\eqref{i: (G, Gamma, H, widehat X_0, hat pi_0)}, $(G,\Gamma,H,\widehat X_0,\hat\pi_0)$ is called the Molino's description of $X$; in particular, $G$ is called the structural local group according to \cite{Molino1988,AlvarezMoreira2016}, and $H$ is called the discriminant group according to \cite{DyerHurderLukina2016}. In~\eqref{i: C^infty Molino's description}, a foliated space is said to be $C^\infty$ when it has a foliated atlas whose changes of coordinates are $C^\infty$ along the leaves, and their leafwise partial derivatives of arbitrary order are continuous (on the ambient space). Other related concepts are defined in the same way, like $C^\infty$ foliated maps, $C^\infty$ diffeomorphisms, (leafwise) tangent space, (leafwise) Riemannian metrics, (leafwise) Riemannian foliated spaces, etc. The concept of $C^\infty$ foliated homogeneity is defined like foliated homogeneity using $C^\infty$ foliated diffeomorphisms.

Our second goal is to characterize $G$-foliated spaces using ($C^\infty$) foliated homogeneity, showing the following results.

\begin{mainthm}\label{mt: foliated homogeneous => G-foliated space}
	If a foliated space $X$ is compact, minimal and foliated homogeneous, then it satisfies hypotheses of Theorem~\ref{mt: Molino} and is a $G$-foliated space for some locally compact local group $G$.
\end{mainthm}

\begin{mainthm}\label{mt: C^infty foliated homogeneous <=> G-fol sp}
	 Suppose that a foliated space $X$ is compact, minimal and $C^\infty$. Then the following conditions are equivalent:
	 	\begin{enumerate}
	 
	 		\item\label{i: C^infty foliated homogeneous} $X$ is $C^\infty$ foliated homogeneous.
			
			\item\label{i: foliated homogeneous} $X$ is foliated homogeneous. 
			
			\item\label{i: G-foliated sp} $X$ satisfies the hypotheses of Theorem~\ref{mt: Molino} and is a $G$-foliated space for some locally compact local group $G$.
		\end{enumerate}
\end{mainthm}

Theorem~\ref{mt: foliated homogeneous => G-foliated space} follows with an adaptation of an argument of Clark and Hurder \cite[Theorem~5.2]{ClarkHurder2013}, using that the canonical left action of the group of foliated homeomorphisms is micro-transitive by a theorem of Effros \cite{Effros1965,vanMill2004}.

To prove Theorem~\ref{mt: C^infty foliated homogeneous <=> G-fol sp}, it is enough to show ``$\text{\eqref{i: G-foliated sp}}\Rightarrow\text{\eqref{i: C^infty foliated homogeneous}}$'' by Theorem~\ref{mt: foliated homogeneous => G-foliated space}. Assuming~\eqref{i: G-foliated sp}, we get the so-called structural right local transverse action, which has its own interest; for instance, it was introduced and used in \cite{AlvarezKordyukov2008a} for Lie foliations. It is the unique ``foliated right local action up to leafwise homotopies'' of $G$ on $X$, which corresponds to the local right translations on $G$ via foliated charts (Proposition~\ref{p: exists a right local transverse action} and Section~\ref{s: structural local transverse action}). Its construction uses a partition of unity subordinated to a foliated atlas and the leafwise center of mass for some (leafwise) Riemannian metric to merge the obvious right local transverse actions on the domains of foliated charts. The structural right local transverse action gives~\eqref{i: C^infty foliated homogeneous} because we always have leafwise homogeneity (Proposition~\ref{p: leafwise homogeneous}).

In Theorem~\ref{mt: C^infty foliated homogeneous <=> G-fol sp}, our proof of  ``$\text{\eqref{i: G-foliated sp}}\Rightarrow\text{\eqref{i: C^infty foliated homogeneous}}$'' needs the $C^\infty$ structure of $X$ because we use the leafwise center of mass as an auxiliary tool. Of course, it could be possible to avoid the $C^\infty$ condition and show ``$\text{\eqref{i: G-foliated sp}}\Rightarrow\text{\eqref{i: foliated homogeneous}}$'' directly with other tools, but that procedure would certainly require more work.

Since there exist leaves without holonomy, and since the (differentiable) quasi-isometry type of the leaves is independent of the choice of a (leafwise) Riemannian metric on $X$, it follows that $X$ is not foliated homogeneous if there is a leaf with holonomy, or if there is a pair of non-quasi-isometric leaves. The converse statement is not true in general.  Fokkink and Oversteegen \cite[Theorem~35]{FokkinkOversteegen2002} constructed an example of a non-homogeneous weak solenoid all of whose leaves are simply connected, and therefore it has no holonomy, and its leaves are quasi-isometric to each other because weak solenoids are suspension foliated spaces. Dyer, Hurder and Lukina constructed more examples of such weak solenoids \cite[Theorem~10.7]{DyerHurderLukina2017}. In Section~\ref{ss: foliated homogeneity may not be told by the leaves}, we give an example of a compact foliated space $X$ satisfying the conditions of Theorem~\ref{mt: Molino}, which is not foliated homogeneous and has no holonomy, whose leaves are quasi-isometric to each other, and with locally connected local transversals (thus it is not a weak solenoid).

\begin{ack}
We thank the referee for many suggestions that have improved the paper.
\end{ack}

\section{Preliminaries}\label{s: prelim}

See \cite[Chapter~II]{MooreSchochet1988}, \cite{Ghys2000} and \cite[Chapter~11]{CandelConlon2000-I} for the needed preliminaries on foliated spaces and interesting examples, and \cite{Haefliger1980,Haefliger1985,Haefliger1988} for the preliminaries on pseudogroups. We mainly follow \cite[Sections~2 and~4A]{AlvarezMoreira2016}, which in turn follows \cite{AlvarezCandel2009,AlvarezCandel2010,AlvarezCandel2018}. Some ideas are also taken from \cite{ClarkHurder2013,AlvarezMasa2008,AlvarezMasa2006}. The needed basic concepts and tools are recalled here for the reader's convenience, and a few new observations are also made.

In the whole paper, unless otherwise stated, spaces are assumed to be locally compact and Polish, and maps are assumed to be continuous. In particular, this applies to foliated spaces, topological groups, local groups and partial maps.

\subsection{Pseudogroups}\label{ss: pseudogrs}

For spaces $T$ and $T'$, the notation $\phi:T\rightarrowtail T'$ is used for a partial map. We will only consider the case where its domain, $\dom\phi$, is open in $T$. The germ of $\phi$ at any $u\in\dom\phi$ will be denoted by $\germ(\phi,u)$. If $\phi$ is an open embedding, we may identify $\phi$ with the homeomorphism $\phi:\dom\phi\to\im\phi$ of an open subset of $T$ to an open subset of $T'$, whose inverse can be considered as a partial map with open domain, $\phi^{-1}:T'\rightarrowtail T$; in particular, when $T=T'$, such a $\phi$ is called a {\em local transformation\/} of $T$.

Given another space $T''$, let $\Phi$ and $\Psi$ be families of partial maps $T\rightarrowtail T'$ and $T'\rightarrowtail T''$, respectively, with open domains. We use the notation $\Psi\Phi=\{\,\psi\phi\mid\phi\in\Psi,\ \psi\in\Psi\,\}$; in particular, $\Phi^n=\Phi\cdots\Phi$ ($n$ times) if $T=T'$ and $n\in\Z^+$. If $\Phi$ consists of open embeddings, let $\Phi^{-1}=\{\,\phi^{-1}\mid\phi\in\Phi\,\}$.

Recall that a {\em pseudogroup\/} $\HH$ on $T$ is a family of local transformations of $T$ that contains $\id_T$, and is closed by the operations of composition, inversion, restriction to open sets and union. It is said that $\HH$ is {\em generated\/} by $S\subset\HH$ if $\HH$ can be obtained from $S$ using the above operations. By considering a pseudogroup as a direct generalization of a group of transformations, the basic dynamical concepts have obvious generalizations to pseudogroups, like {\em orbits\/}, {\em saturation\/}, ({\em topological\/}) {\em transitivity\/} and {\em minimality\/}. The orbit space is denoted by $T/\HH$. The $\HH$-saturation of any $A\subset T$ is denoted by $\HH(A)$, and the orbit of any $u\in T$ by $\HH(u)$. For any open $V\subset T$, the {\em restriction\/} $\HH|_V:=\{\,h\in\HH\mid \dom h,\im h\subset V\,\}$ is a pseudogroup.

Given another pseudogroup $\HH'$ on $T'$, a {\em morphism\/} $\Phi\colon\HH\to\HH'$ is a maximal collection of partial maps $T\rightarrowtail T'$ with open domain such that $\HH'\Phi\HH\subset\Phi$, $T=\bigcup_{\phi\in\Phi}\dom\phi$, and, for all $\phi,\psi\in\Phi$ and $u\in\dom\phi\cap\dom\psi$, there is some $h'\in\HH'$ so that $\phi(u)\in\dom h'$ and $\germ(h'\phi,u)=\germ(\psi,u)$. Let $\Phi_0$ be a family of partial maps $T\rightarrowtail T'$ with open domain such that $T=\HH(\bigcup_{\phi\in\Phi}\dom\phi)$, and there is a subset $S$ of generators of $\HH$ such that, if $\phi,\psi\in\Phi_0$, $h\in S$ and $u\in\dom\phi\cap\dom\psi h$, then there is some
$h'\in\HH'$ so that $\phi(u)\in\dom h'$ and $\germ(h'\phi,u)=\germ(\psi h,u)$. Then there is a unique morphism $\Phi:\HH\to\HH'$ containing $\Phi_0$, which is said to be {\em generated\/} by $\Phi_0$. 

With the terminology of Haefliger \cite{Haefliger1980,Haefliger1985,Haefliger1988}, an {\em \'etal\'e morphism\/} $\Phi:\HH\to\HH'$ is a maximal family of homeomorphisms of open subsets of $T$ to open subsets of $T'$ such that $\HH'\Phi\HH\subset\Phi$, $T=\bigcup_{\phi\in\Phi}\dom\phi$ and $\Phi\Phi^{-1}\subset\HH'$. If moreover $\Phi^{-1}$ is an \'etal\'e morphism, then $\Phi$ is called an {\em equivalence\/}, and the pseudogroups $\HH$ and $\HH'$ are said to be {\em equivalent\/}. If $\Phi_0$ is a family of homeomorphisms of open subsets of $T$ to open subsets of $T'$ such that $T=\HH(\bigcup_{\phi\in\Phi}\dom\phi)$ and $\Phi_0\HH\Phi_0^{-1}\subset\HH'$, then there is a unique \'etal\'e morphism $\Phi:\HH\to\HH'$ containing $\Phi_0$, which is said to be {\em generated\/} by $\Phi_0$. Equivalent pseudogroups are considered to have the same dynamics. For instance, $\HH$ is equivalent to $\HH|_V$ for any open $V\subset T$ that meets all $\HH$-orbits. In fact, $\Phi:\HH\to\HH'$ is an equivalence if and only if $\GG=\HH\cup\HH'\cup\Phi\cup\Phi^{-1}$ is a pseudogroup on $T\sqcup T'$ such that $T$ and $T'$ meet all $\GG$-orbits, $\GG|_T=\HH$ and $\GG|_{T'}=\HH'$.

The germs $\germ(h,u)$, for $h\in\HH$ and $u\in\dom h$, form a topological groupoid $\fH$, equipped with the sheaf topology and the operation induced by composition. Its unit subspace can be identified with $T$. In fact, $\fH$ is an \'etal\'e groupoid (the source and target maps, $s,t:\fH\to T$, are local homeomorphisms). Given $x\in T$, the group of elements of $\gamma\in\fH$ with $s(\gamma)=t(\gamma)=x$ is called the {\em germ group\/} of $\HH$ at $x$.

Let us recall the following definitions of properties that $\HH$ may have:
	\begin{description}
	
		\item[Compact generation] This means that there is a relatively compact open $U\subset T$, which meets all orbits, such that $\HH|_U$ is generated by a finite set, $E=\{h_1,\dots,h_k\}$, and every $h_i$ has an extension $\tilde h_i\in\HH$ with $\ol{\dom h_i}\subset\dom\tilde h_i$. This $E$ is called a {\em system of compact generation\/} of $\HH$ on $U$.
		
		\item[(Strong) equicontinuity] This means that there are an open cover $\{T_i\}$ of $T$ and a metric $d_i$ inducing the topology of every $T_i$, and $\HH$ is generated by some subset $S\subset\HH$, with $S^2\subset S=S^{-1}$ ($S$ is symmetric and closed by compositions\footnote{The term {\em pseudo$*$group\/} was used in \cite{AlvarezMoreira2016} when these conditions are satisfied. This term was introduced in \cite{Matsumoto2010} for a family that moreover contains $\id_T$ and is also closed by restrictions to open subsets.}), such that, 
for every $\epsilon>0$, there is some $\delta>0$ so that 
			\[
				d_i(x,y)<\delta\Longrightarrow d_j(h(x),h(y))<\epsilon 
			\]
		for all $h\in S$, indices $i,j$, and $x,y\in T_i\cap h^{-1}(T_j\cap\im h)$.
		
		\item[Strong quasi-analyticity] This means that $\HH$ is generated by some subset $S\subset\HH$, with $S^2\subset S=S^{-1}$, such that, if any $h\in S$ is the identity on some non-empty open subset of its domain, then $h=\id_{\dom h}$.
		
		\item[Strong local freeness] This means that $\HH$ is generated by some subset $S\subset\HH$, with $S^2\subset S=S^{-1}$, such that, if any $h\in S$ fixes some point in its domain, then $h=\id_{\dom h}$. Equivalently, this means that $\HH$ is strongly quasi-analytic and all of its germ groups are trivial.
		
	\end{description}
These properties are invariant by equivalences. If compact generation holds with some $U$, then it also holds with any other relatively compact open subset of $T$ that meets all orbits. Let $\PP$ denote any of the above last three properties. If $\PP$ holds with $S$, then it also holds with its {\em localization\/},
	\[
		S_{\text{\rm loc}}=\{\,h|_O\mid h\in S,\  \text{$O$ is open in $\dom h$}\,\}\;.
	\]
Moreover we can add $\id_T$ to $S$ if desired (obtaining $S^2=S$). If $\HH$ is compactly generated and satisfies $\PP$, then, for every relatively compact open $U\subset T$ that meets all orbits, we can choose a system of compact generation $E$ of $\HH$ on $U$ such that $\HH|_U$ also satisfies $\PP$ with $S=\bigcup_{n=1}^\infty E^n$. The following result lists some needed non-elementary properties.

\begin{prop}[{\cite[Proposition~8.9, and Theorems~11.1 and~12.1]{AlvarezCandel2009}, \cite{Tarquini2004} and \cite[Theorems 3.3 and 5.2]{AlvarezCandel2010}}]\label{p: list, HH}
Suppose that $\HH$ is compactly generated, equicontinuous and strongly quasi-analytic. Then the following holds:
	\begin{enumerate}
		
		\item\label{i: VV} Assume that $\HH$ satisfies the condition of compact generation with $U$, $E=\{h_1,\dots,h_k\}$ and $\tilde h_1,\dots,\tilde h_k$. For every $h=h_{i_n}\cdots h_{i_1}\in\bigcup_{n=1}^\infty E^n$, let $\tilde h=\tilde h_{i_n}\cdots\tilde h_{i_1}$. Then there is a finite family $\VV$ of open subsets of $T$ covering $U$ such that, for any $h\in\bigcup_{n=1}^\infty E^n$ and $V\in\VV$, we have $V\subset\dom\tilde h$ if $V\cap\dom h\ne\emptyset$.
		
		\item\label{i: ol HH} Suppose that $\HH$ satisfies the equicontinuity condition with a set $S$. Then $\ol{C(O,T)\cap S_{\text{\rm loc}}}$ consists of local transformations for all small enough open subsets $O\subset T$, where the closure is taken in the compact-open topology, and the pseudogroup $\ol\HH$ generated by such transformations is equicontinuous. More precisely, $\ol\HH$ satisfies the equicontinuity condition with the set $\ol S$ determined by the condition $C(O,T)\cap\ol S=\ol{C(O,T)\cap S_{\text{\rm loc}}}$ for all $O$ as above. 
		
		\item\label{i: transitive <=> minimal} The orbit closures are minimal sets, and therefore $\HH$ is transitive if and only if it is minimal.
		
	\end{enumerate}
\end{prop}

In Proposition~\ref{p: list, HH}-\eqref{i: ol HH}, the pseudogroup $\ol\HH$ is called the {\em closure\/} of $\HH$.

\subsection{Relation of pseudogroups with local groups and local actions}\label{ss: local actions}

The general definition of local group is rather involved \cite{Jacoby1957}, but, in the locally compact case, a {\em local group\/} $G$ can be considered as a neighborhood of the identity element $e$ in some topological group \cite{DriesGoldbring2010,DriesGoldbring2012}; this description can be used as definition. Two such neighborhoods in the same topological group define {\em equivalent\/} local groups; thus it can be said that, up to equivalences, an equivalence class of local groups is the ``germ'' of a topological group at the identity element. For the sake of simplicity, the family of open neighborhoods of $e$ in $G$ will be denoted by $\NN(G,e)$. Given another local group $G'$ with identity element $e'$, a {\em local homomorphism\/} of $G$ to $G'$ is a partial map with open domain, $\sigma:G\rightarrowtail G'$, such that $e\in\dom\sigma$, $\sigma(e)=e'$, and $\sigma(gh)=\sigma(g)\sigma(h)$ for all $g,h\in\dom\sigma$ such that the products $gh$ and $\sigma(g)\sigma(h)$ are defined with $gh\in\dom\sigma$. Two local homomorphisms of $G$ to $G'$ are {\em equivalent\/} when they have the same germ at $e$. If there is a local homomorphism $\tau:G'\rightarrowtail G$ such that $\tau\sigma$ and $\sigma\tau$ are equivalent to $\id_G$ and $\id_{G'}$, then $\sigma$ is called a {\em local isomorphism\/}. 

The term {\em sublocal group\/} will be used for a subspace $H\subset G$ such that $(H\cap V)^2,(H\cap V)^{-1}\subset H$ for some $V\in\NN(G,e)$; in particular, $e\in H$, but $H\cap V$ is not required to be closed in $V$ (contrary to \cite[Definition~2.10.]{Goldbring2010}). A sublocal group becomes a local group with the induced structure, but it may not be locally compact, and the inclusion map of any sublocal group is a local homomorphism.

A {\em right local action\/} of $G$ on $T$ is a partial map with open domain, $\chi:T\times G\rightarrowtail T$, where $T\times\{e\}\subset\dom\chi$ and $\chi(u,e)=u$ for all $u\in T$, and such that, for all $g,h\in G$ and $u\in T$, if the product $gh$ is defined and $(u,g),(u,gh),(\chi(u,g),h)\in\dom\chi$, then $\chi(\chi(u,g),h)=\chi(u,gh)$. Two right local actions of $G$ on $T$ are {\em equivalent\/} when they agree around $T\times\{e\}$. If $T$ is compact, we can assume $\dom\chi=T\times O$ for some $O\in\NN(G,e)$. For any open $V\subset T$, the restriction $\chi:\chi^{-1}(V)\cap(V\times G)\to V$ is a right local action of $G$ on $V$, called the restriction of $\chi$ to $V$. Given an open cover $\{T_i\}$ of $T$ and a right local action $\chi_i$ of $G$ on every $T_i$ such that the restrictions of $\chi_i$ and $\chi_j$ to $T_i\cap T_j$ are equivalent, it is easy to check that there is a unique right local action of $G$ on $T$, up to equivalences, whose restriction to every $T_i$ is equivalent to $\chi_i$.

Consider another right local action $\chi'$ of $G'$ on $T'$. A partial map with open domain, $\phi:T\rightarrowtail T'$, is called {\em locally equivariant\/} if there is some open neighborhood $\Sigma$ of $\dom\phi\times\{e\}$ in $\dom\chi\cap(\phi\times\id_G)^{-1}(\dom\chi')$ such that $\chi(\Sigma)\subset\dom\phi$ and $\phi\chi(u,g)=\chi'(\phi(u),g)$ for all $(u,g)\in\Sigma$. Note that compositions, restrictions to open sets and unions of locally equivariant partial maps with open domain are locally equivariant, as well as their inverses whenever defined. A family of partial maps $T\rightarrowtail T'$ with open domain is called {\em locally equivariant\/} when all of its elements are locally equivariant.

{\em Local anti-homomorphisms\/}, {\em left local actions\/}, their {\em equivalences\/} and corresponding {\em locally equivariant\/} maps are similarly defined.

For instance, any finite dimensional metrizable locally compact local group is indeed locally isomorphic to the direct product of a Lie group and a compact zero-dimensional topological group \cite[Theorem~107]{Jacoby1957} (corrected according to \cite{Goldbring2010}, or using \cite{DriesGoldbring2010,DriesGoldbring2012} and \cite[Section~IV.4.9]{MontgomeryZippin1955}). As a concrete example, we can consider the product of any local Lie group and any countable family of finite groups. By Ado's theorem, the equivalence classes of local Lie groups and their local homomorphisms correspond one-to-one to finite dimensional real Lie algebras and their homomorphisms. If $G$ is a profinite group, $\Gamma$ is a dense subgroup of $G$, and $H$ is an open neighborhood of the identity in $G$, then $G$ and $H$ are locally isomorphic local groups, and $H$ and $\Gamma\cap H$ are sublocal groups of $G$ and $H$. A typical example of right local action of a local group $G$ on itself is given by its local right translations, and any local left translation of $G$ becomes locally equivariant.

\begin{prop}[{\cite[Theorems~3.3 and~5.2]{AlvarezCandel2010}, \cite[Lemma~2.36, Theorem~2.38 and Remark~21]{AlvarezMoreira2016}}]\label{p: list, G}
The following holds:
	\begin{enumerate}
		
		\item\label{i: G} Suppose that $\HH$ is minimal, compactly generated, equicontinuous and strongly quasi-analytic. Then $\ol\HH$ is strongly locally free if and only if $\HH$ is equivalent to a pseudogroup $\GG$ on some local group $G$ generated by the left local action by local left translations of a finitely generated dense sublocal group $\Gamma\subset G$. 
		
		\item\label{i: G >--> G'} Let $\GG$ and $\GG'$ be the pseudogroups on local groups $G$ and $G'$ generated by the left local actions by local left translations of respective finitely generated dense sublocal groups $\Gamma$ and $\Gamma'$. Let $\Phi:\GG\to\GG'$ be a morphism such that $\GG(e)\mapsto\GG'(e')$ by the induced map $G/\GG\to G/\GG'$. Then $\Phi$ is generated by a local homomorphism $G\rightarrowtail G'$ that restricts to a local homomorphism $\Gamma\rightarrowtail\Gamma'$.
		
	\end{enumerate}
\end{prop}

\begin{prop}\label{p: chi'}
	Let $\Phi:\HH\to\HH'$ be an equivalence between compactly generated pseudogroups. Let $\chi$ be a right local action of $G$ on $T$ such that $\HH$ is locally equivariant. Then there is a unique right local action $\chi'$ of $G$ on $T'$, up to equivalences, such that $\Phi$ and $\HH'$ are locally equivariant.
\end{prop}

\begin{proof}
	For $x\in T'$, $\phi\in \Phi$ with $x\in \im(\Phi)$, and $g$ sufficiently close to $e$ in $G$, our tentative right action $\chi'$ will satisfy $\chi'(x,g)= \phi(\chi(\phi^{-1}(x),g))$. The proof is then a matter of verifying that this expression determines a well-defined local right action satisfying the hypothesis. We omit the details for brevity.
\end{proof}

Let $\chi$ be a right local action of $G$ on $T$ such that $\HH$ is locally equivariant. Consider the following property that $(T,\HH,\chi)$ may have:
	\begin{equation}\label{HH(chi(u times P)) = T}
		\HH(\chi(\{u\}\times P))=T\quad\forall u\in T, \forall P\in\NN(G,e)\mid\{u\}\times P\subset\dom\chi.
	\end{equation}

\begin{lem}\label{l: HH(chi(u times P)) = T}
	Property~\eqref{HH(chi(u times P)) = T} is preserved by locally equivariant pseudogroup equivalences.
\end{lem}

\begin{proof}
	Elementary.
\end{proof}

\subsection{Foliated spaces}\label{ss: foliated sps}

The notation introduced here will be used in Sections~\ref{s: Molino}--\ref{s: C^infty G-fol sp => C^infty foliated homogeneous}.

Let $X$ be a space and $n\in\Z^{\ge0}$. The main results of the paper will require $X$ to be compact, but this condition is avoided for the basic concepts. Let $\UU$ be a family consisting of pairs $(U_i,\xi_i)$, called {\em foliated charts\/}, where $\{U_i\}$ is an open cover of $X$, and every $\xi_i$ is a homeomorphism $U_i\to B_i\times T_i$ for some contractible open subset $B_i\subset\R^n$ and a space $T_i$. Every $(U_i,\xi_i)$ induces a projection $p_i:U_i\to T_i$ whose fibers are called {\em plaques\/}. Assume that finite intersections of plaques are open in the plaques. Then the open subsets of the plaques form a base of a finer topology in $X$, becoming an $n$-manifold whose connected components are called {\em leaves\/}. In this case, it is said that $\UU$ defines a {\em foliated structure\/} $\FF$ of {\em dimension\/} $n$ on $X$, $X\equiv(X,\FF)$ is called a {\em foliated space\/} (or {\em lamination\/}), and $\UU$ is called a {\em foliated atlas\/}. Two foliated atlases define the same foliated structure if their union is a foliated atlas. The subspaces $\xi_i^{-1}(\{\bv\}\times T_i)\subset X$, $\bv\in B_i$, are called {\em local transversals\/} defined by the foliated chart $(U_i,\xi_i)$. A {\em transversal\/} is a subspace $\Sigma\subset X$ where any point has a neighborhood that is a local transversal of some foliated chart. A transversal is called {\em global\/} if it meets all leaves.

A foliated space can be considered as a weak version of a regular dynamical system where the leaves play the role of the orbits. In this way, several basic dynamical concepts have obvious versions for foliated spaces, like {\em saturation\/}, ({\em topological\/}) {\em transitivity\/} and {\em minimality\/}. The partition of $X$ into leaves is enough to describe $\FF$. The leaf through a point $x$ may be denoted by $L_x$, and the leaf space by $X/\FF$. The saturation of a subset $A\subset X$ is denoted by $\FF(A)$.

We can assume that the foliated atlas $\UU$ is {\em regular\/}\footnote{Regularity of the foliated atlas is used with another meaning in \cite{ClarkHurder2013}.} in the sense that it satisfies the following properties \cite[Definition~5.1]{AlvarezCandel2018} (see also \cite{HectorHirsch1986-A,CandelConlon2000-I,Godbillon1991}):
	\begin{itemize}
 
		\item there is another foliated atlas $\widetilde\UU=\{\widetilde U_i,\tilde \xi_i\}$ of $X$, with $\tilde\xi_i:\widetilde U_i\to\widetilde B_i\times\widetilde T_i$ and distinguished submersions $\tilde p_i:U_i\to\widetilde T_i$, such that $\ol{U_i}\subset\widetilde U_i$, $\ol{B_i}\subset\widetilde B_i$, $T_i$ is an open subspace of $\widetilde T_i$, and $\xi_i=\tilde\xi_i|_{U_i}$ (thus $p_i=\tilde p_i|_{U_i}$);

		\item $\{U_i\}$ is locally finite; and

		\item every plaque of $(U_i, \xi_i)$ meets at most one plaque of $(U_j, \xi_j)$.

	\end{itemize}
By the last condition, there are homeomorphisms\footnote{This convention for the order of these subindices agrees with \cite{ClarkHurder2013} and differs from \cite{AlvarezMoreira2016}. The same kind of convention will be used in the local representations of foliated maps.} $h_{ij}:p_j(U_i\cap U_j)\to p_i(U_i\cap U_j)$, the {\em elementary holonomy transformations\/}, such that $h_{ij}p_j=p_i$ on $U_i\cap U_j$, obtaining the {\em defining cocycle\/} $\{U_i,p_i,h_{ij}\}$; it describes $\FF$ and satisfies the cocycle condition $h_{ik}=h_{ij}h_{jk}$ on $p_k(U_i\cap U_j\cap U_k)$. So the changes of coordinates $\xi_i\xi_j^{-1}:\xi_j(U_i\cap U_j)\to\xi_i(U_i\cap U_j)$ are of the form
	\begin{equation}\label{xi_i xi_j^-1}
		\xi_i\xi_j^{-1}(\bv,u)=(g_{ij}(\bv,u),h_{ij}(u))\;,
	\end{equation}
for some maps $g_{ij}:\xi_j(U_i\cap U_j)\to B_i$. 

The ``transverse dynamics'' of $X$ are described by its {\em holonomy pseudogroup\/}, which is (the equivalence class of) the pseudogroup $\HH$ generated by the maps $h_{ij}$ on $T:=\bigsqcup_iT_i$. Its elements are called {\em holonomy transformations\/}. There is a canonical identity $X/\FF\equiv T/\HH$, where the $\HH$-orbit that corresponds to a leaf $L$ is $\bigcup_ip_i(L\cap U_i)$. Via this identity, $\FF$-leaves and $\HH$-orbits have corresponding dynamical concepts.

We can assume that $\widetilde\UU$ is also regular, obtaining elementary holonomy transformations $\tilde h_{ij}:\tilde p_j(\widetilde U_i\cap\widetilde U_j)\to\tilde p_i(\widetilde U_i\cap\widetilde U_j)$, extending the maps $h_{ij}$, which generate another representative of the holonomy pseudogroup, $\widetilde\HH$ on $\widetilde T:=\bigsqcup_i\widetilde T_i$; $T$ is an open subspace of $\widetilde T$ that meets all $\widetilde\HH$-orbits, and $\HH=\widetilde\HH|_T$. Let $\sigma_i:T_i\to U_i$ and $\tilde\sigma_i:\widetilde T_i\to\widetilde U_i$ be the sections of every $p_i$ and $\tilde p_i$ defined by fixing an element of $B_i$ (thus $\sigma_i=\tilde\sigma_i|_{T_i}$). We can assume that the sets $\tilde\sigma_i(\widetilde T_i)$ are separated by open sets in $X$, and therefore $\bigcup_i\tilde\sigma_i:\widetilde T\to\bigcup_i\tilde\sigma_i(\widetilde T_i)$ and $\bigcup_i\sigma_i:T\to\bigcup_i\sigma_i(T_i)$ are homeomorphisms to complete transversals.

Given a finite sequence of indices, $\II=(i_0,\dots,i_\alpha)$, let $h_\II=h_{i_\alpha i_{\alpha-1}}\cdots h_{i_1i_0}$ if $\alpha>0$, and $h_\II=\id_{T_{i_0}}$ if $\alpha=0$. If $\dom h_\II\ne\emptyset$, then $\II$ is called {\em admissible\/}. Let $c:I:=[0,1]\to X$ be a path from $x$ to $y$, which is {\em leafwise\/} in the sense that $c(I)$ is contained in some leaf $L$. Let us say that $c$ is ($\UU$-) {\em covered\/} by $\II$ if there is a partition of $I$, $0=t_0<t_1<\dots<t_{\alpha+1}=1$, such that $c([t_k,t_{k+1}])\subset U_{i_k}$ for all $k=0,\dots,\alpha$. In this case, $u:=p_{i_0}(x)\in\dom h_\II$ and $h_\II(u)=p_{i_\alpha}(y)$. If $\II=(i_0,\dots,i_\alpha)$ and $\JJ=(j_0,\dots,j_\beta)$ cover $c$ and $c'$, respectively, with $j_0=i_\alpha$, then $\II\JJ:=(i_0,\dots,i_\alpha=j_0,\dots,j_\beta)$ and $\II^{-1}:=(i_\alpha,\dots,i_0)$ cover $cc'$ and $c^{-1}$, respectively, and we have $h_{\II\JJ}=h_\JJ h_\II$ and $h_{\II^{-1}}=h_\II^{-1}$. By using $\widetilde\UU$, we can similarly define $\tilde h_\II$, which is an extension of $h_\II$. Recall that, for another admissible sequence $\JJ=(j_0,\dots,j_\beta)$ with $j_0=i_0$ and $j_\beta=i_\alpha$, covering another path $c'$ from $x$ to $y$ in $L$, if $c$ and $c'$ are endpoint-homotopic in $L$, then $u\in\dom h_\JJ$ and $\germ(h_\II,u)=\germ(h_\JJ,u)$. Any leafwise path is covered by some admissible sequence, and, vice versa, for all $\II=(i_0,\dots,i_\alpha)$, $x\in U_{i_0}$ and $y\in U_{i_\alpha}$ with $p_{i_0}(x)\in\dom h_\II$ and $h_\II p_{i_0}(x)=p_{i_\alpha}(y)$, there is some leafwise path from $x$ to $y$ covered by $\II$.

The {\em holonomy group\/} $\Hol(L,x)$ of a leaf $L$ at a point $x\in L\cap U_i$ is the germ group of $\HH$ at $u=p_i(x)$. It depends only on $L$ up to conjugation by germs of holonomy transformations. The {\em holonomy homomorphism\/}, $\hol:\pi_1(L,x)\to\Hol(L,x)$, is given by $\hol([c])=\germ(h_\II^{-1},u)$ if $c$ is covered by $\II=(i_0,\dots,i_\alpha)$ with $i_0=i_\alpha=i$. This epimorphism induces a regular covering $\widetilde L^{\text{\rm hol}}$ of $L$, the {\em holonomy covering\/}. We will consider the canonical right action of $\Hol(L,x)$ on $\widetilde L^{\text{\rm hol}}$ by covering transformations. A leaf is said to be {\em without holonomy\/} if its holonomy group is trivial, and $X$ is called {\em without holonomy\/} when all leaves have no holonomy. The union of leaves without holonomy is a dense $G_\delta$ in $X$, and therefore Borel and residual \cite{Hector1977a,EpsteinMillettTischler1977}. A path connected subset of a leaf, $D\subset L$, is said to be {\em without holonomy\/} if the composition
	\[
		\begin{CD}
			\pi_1(D,x) @>>> \pi_1(L,x)@>{\hol}>>\Hol(L,x)
		\end{CD}
	\]
is trivial for some (and therefore all) $x\in D$.

It is said that $X$ is ({\em strongly\/}) {\em equicontinuous\/}, {\em strongly quasi-analytic\/} or {\em strongly locally free\/} if $\HH$ satisfies these properties; thus $X$ is strongly locally free just when it is strongly quasi-analytic and has no holonomy. In the definition of these conditions for $\HH$, by refining $\UU$ if necessary, we can assume that the metrics $d_i$ are defined on the sets $T_i$, and we can take
	\[
		S=\{\,h_\II\mid\text{$\II$ is an admissible sequence}\,\}
	\]
if desired. 

For a local group $G$, we say that $X$ is a {\em $G$-foliated space\/} if $\HH$ is equivalent to a pseudogroup generated by some local left translations on $G$. More precisely, this means that $X$ has a regular foliated atlas $\UU$, with induced distinguished submersions $p_i:U_i\to T_i$ and elementary holonomy transformations $h_{ij}:p_j(U_i\cap U_j)\to p_i(U_i\cap U_j)$, such that every $T_i$ is an open subspace of $G$, and every $h_{ij}$ is the restriction of some local left translation of $G$.
	
If $X$ is compact, then $\UU$ is finite and $T$ is relatively compact in $\widetilde T$, obtaining that $\widetilde\HH$ satisfies the definition of compact generation with the generators $h_{ij}$ of $\widetilde\HH|_T=\HH$ and their extensions $\tilde h_{ij}$. So $\HH$ is also compactly generated. If moreover $\FF$ is equicontinuous, then the properties of Propositions~\ref{p: list, HH} and~\ref{p: list, G} apply to $\HH$; in particular, the leaf closures are minimal sets, and therefore $X$ is transitive if and only if it is minimal.

{\em Foliated spaces with boundary\/} can be defined in a similar way, adapting the definition of manifold with boundary: every $B_i$ would be a contractible open set in the half space $H^n\equiv\R^{n-1}\times[0,\infty)$. The {\em boundary\/} of $X$, $\partial X=\bigcup_i\xi_i^{-1}(\partial B_i\times T_i)$, becomes a foliated space without boundary. The basic concepts recalled here about foliated spaces have direct extensions to foliated spaces with boundary.

Any open $U\subset X$ becomes a foliated space with the {\em restriction\/} $\FF|_U$, defined by all possible foliated charts of $\FF$ with domain in $U$. Any connected (second countable) manifold can be considered a foliated space with one leaf, and any space can be considered as a foliated space whose leaves are the points. Like in the case of foliations, a typical example of foliated space can be obtained by {\em suspension\/} of an action of the fundamental group of a manifold on a space (see Section~\ref{ss: suspensions}). Other interesting examples of foliated spaces are {\em weak solenoids\/}, defined as inverse limits of towers of covering maps between closed manifolds. More general foliated spaces are given by inverse limits of towers of foliated covering maps between closed foliated manifolds (Section~\ref{ss: inverse limits of Lie foliations}).

Let $X'\equiv(X',\FF')$ be another foliated space, let $\UU'=\{U'_a,\xi'_a\}$ be a regular foliated atlas of $X'$, where $\xi'_a:U'_a\to B'_a\times T'_a$, giving rise to a defining cocycle $\{U'_a,p'_a,h'_{ab}\}$, and the corresponding representative of the holonomy pseudogroup, $\HH'$ on $T':=\bigsqcup_aT_a$ generated by $\{h'_{ab}\}$. A map $\phi:X\to X'$ is called {\em foliated\/} when it maps leaves to leaves. Then every local representation $\xi'_a\phi\xi_i^{-1}:\xi_i(U_i\cap\phi^{-1}(U'_a))\to B'_a\times T'_a$ is of the form
	\begin{equation}\label{xi'_a phi xi_i^-1}
		\xi'_a\phi\xi_i^{-1}(\bv,u)=(\phi^1_{ai}(\bv,u),\phi^2_{ai}(u))
	\end{equation}
for some maps $\phi^1_{ai}:\xi_i(U_i\cap\phi^{-1}(U'_a))\to B'_a$ and $\phi^2_{ai}:p_i(U_i\cap\phi^{-1}(U'_a))\to T'_a$. The maps $\phi^2_{ai}$ generate a morphism $\Phi:\HH\to\HH'$ \cite{AlvarezMasa2006,AlvarezMasa2008}, which is said to be {\em induced\/} by $\phi$.

 An action of a group on $X$ is called {\em foliated\/} when it is given by foliated homeomorphisms. A homotopy $H$ between foliated maps $\phi,\psi:X\to X'$ is said to be {\em leafwise\/} if it is a foliated map $X\times I\to X'$, where $X\times I$ is endowed with the foliated structure with leaves $L\times I$, for leaves $L$ of $X$; in particular, every path $H(x,\cdot):I\to X'$ ($x\in X$) is leafwise. In this case, $\phi$ and $\psi$ induce the same morphism $\HH\to\HH'$ \cite[Proposition~6.1]{AlvarezMasa2008}. A {\em leafwise\/} isotopy has a similar definition.

Let $V\subset\R^n\times Y$ be an open subset, and let $r\in\Z^{\ge0}\cup\{\infty\}$. A map $g:V\to\R^{n'}$ is called ({\em differentiable of class\/}) $C^r$ when, for any integer $0\le k\le r$ (it is enough to take $k=r$ if $r<\infty$), all partial derivatives of $g$ up to order $k$ with respect to the coordinates of $\R^n$ are defined and continuous on $V$. A change of coordinates $\xi_j\xi_i^{-1}$ is called $C^r$ when the map $g_{ij}$ in~\eqref{xi_i xi_j^-1} is $C^r$. If all changes of coordinates are $C^r$, then $\UU$ defines a {\em $C^r$ structure\/} on $X$, which becomes a {\em $C^r$ foliated space\/}. In this case, $\UU$ and its foliated charts are called $C^r$. Two such foliated atlases of $X$ define the same $C^r$ structure if their union also defines a $C^r$ structure. The leaves of $C^r$ foliated spaces canonically become $C^r$ manifolds. Many concepts of $C^r$ manifolds have straightforward generalizations to $C^r$ foliated spaces, like {\em $C^r$ foliated maps\/}, {\em $C^r$ foliated diffeomorphisms\/}, {\em $C^r$ foliated embeddings\/}, {\em $C^r$ foliated actions\/}, {\em $C^r$ leafwise homotopies/diffeotopies\/}, {\em $C^r$ vector bundles\/}, {\em $C^r$ sections\/}, the ({\em leafwise\/}) {\em tangent bundle\/} $TX$ (or $T\FF$), the ({\em leafwise\/}) {\em tangent map\/} $T\phi:TX\to TX'$ of a $C^r$ foliated map $\phi:X\to X'$, ({\em leafwise\/}) {\em Riemannian metrics\/}, etc. For instance, a foliated map $\phi:X\to X'$ is $C^r$ when, for all local representations $\xi'_a\phi\xi_i^{-1}$, the maps $\phi^1_{ai}$ of~\eqref{xi'_a phi xi_i^-1} are $C^r$.

Any $C^r$ foliated space has a $C^r$ partition of unity subordinated to any open cover \cite[Proposition~2.8]{MooreSchochet1988}. A version of the Reeb's stability theorem holds for $C^2$ foliated spaces \cite[Proposition~1.7]{AlvarezCandel2009}.

Recall that a subset $A$ in a Riemannian manifold $M$ is
  called \emph{convex} when, for all $x,y\in A$, there is a unique 
  minimizing geodesic segment from $x$ to $y$ in $M$ that 
  lies in $A$ (see e.g.\ \cite[Section~IX.6]{Chavel2006}). 
  For example, sufficiently small balls are convex.
If  $X$ is $C^\infty$, given any $C^\infty$ Riemannian metric on $X$, we can choose $\UU$ and $\widetilde\UU$ so that the plaques of their charts are convex balls in the leaves. This follows from the relation between the convexity and injectivity radii \cite[Theorem~IX.6.1]{Chavel2006}, and the continuity of the injectivity radius on closed manifolds \cite{Ehrlich1974,Sakai1983}---the case of closed manifolds easily yields local lower bounds of the injectivity radius on arbitrary manifolds, valid for all metrics that are close enough to a given metric in the weak $C^\infty$ topology.

\subsection{Spaces of foliated maps}\label{ss: spaces of foliated maps}

Suppose that $X$ and $X'$ are $C^r$ for some $r\in\Z^{\ge0}\cup\{\infty\}$. We use the following notation\footnote{The foliated structures are added to this notation to avoid ambiguity.} for sets of maps $X\to X'$: 
	\begin{itemize}
	
		\item $C^r(X,\FF;X',\FF')$ is the set of $C^r$ foliated maps.
		
		\item $\Diffeo^r(X,\FF;X',\FF')$ (or $\Diffeo^r(X,\FF)$ if $X=X'$) is the set of $C^r$ foliated diffeomorphisms.
		
		\item $\Emb^r(X,\FF;X',\FF')$ is the set of $C^r$ foliated embeddings.
		
		\item $\Prop^r(X,\FF;X',\FF')$ is the set of proper $C^r$ foliated maps. 
		
		\item $\Homeo(X,\FF;X',\FF')$ (or $\Homeo(X,\FF)$ if $X=X'$) is the set of foliated homeomorphisms.
		
	\end{itemize}
If $r=0$ or it is clear that $r=\infty$, then $r$ is removed from the above notation. $\Homeo(X,\FF)$ is a subgroup of the group of homeomorphisms, $\Homeo(X)$.

Let us define two foliated versions of the weak/strong $C^r$ topology. In the first version, consider any $\phi\in C^r(X,\FF;X',\FF')$, locally finite families of foliated charts, $\UU=\{U_i,\xi_i\}$ of $X$ and $\UU'=\{U'_a,\xi'_a\}$ of $X'$, a family of compact subsets of $X$, $\KK=\{K_i\}$, so that $K_i\subset U_i$ and $f(K_i)\subset U'_{a_i}$ for all $i$ and corresponding indices $a_i$, a family $\EE=\{\epsilon_i\}$ of positive numbers, and any integer $0\le k\le r$ (it is enough to take $k=r$ if $r<\infty$). Then let $\NN_{\text{\rm F}}^k(\phi,\UU,\UU',\KK,\EE)$ be the set of foliated maps $\psi:X\to X'$ such that $\psi(K_i)\subset U'_{a_i}$ and
	\[
		\left|\frac{\partial^\alpha(\phi^1_{a_ii}-\psi^1_{a_ii})}{\partial\bv^\alpha}(\bv,u)\right|<\epsilon_i\;,
	\]
for all $i$, $(\bv,u)\in\xi_i(K_i)$ and multi-indices $\alpha$ with $|\alpha|\le k$, where $\phi^1_{a_ii}$ and $\psi^1_{a_ii}$ are given by~\eqref{xi'_a phi xi_i^-1}. All possible sets $\NN_{\text{\rm F}}^k(\phi,\UU,\UU',\KK,\EE)$ form a base of open sets in a topology on $C^r(X,\FF;X',\FF')$, called the {\em strong foliated $C^r$ topology\/}. The {\em weak foliated $C^r$ topology\/} is similarly defined by using finite families of indices $i$. The subindex ``WF/SF'' will be added to the notation to indicate that the weak/strong foliated $C^r$ topology in a family of $C^r$ foliated maps. Note that $C_{\text{\rm WF}}(X,\FF;X',\FF')$ has the compact-open topology. Of course both topologies coincide when $X$ is compact, and only the subindex ``F'' will be added in this case. 

If $X$ is compact, then the group of homeomorphisms, $\Homeo(X)$, is a Polish topological group with the compact-open topology \cite[Theorem~3]{Arens1946}. Moreover $\Homeo(X,\FF)$ is a closed subgroup of $\Homeo(X)$, and therefore it is also a Polish topological group.

Some important results on spaces of $C^r$  maps between manifolds have straightforward generalizations to $C^r$ foliated spaces, like the following.

\begin{prop}\label{p: list, SF}
	The following properties hold:
		\begin{enumerate}	
		
			\item\label{i: T_x phi injective/surjective for all x} The injectivity/surjectivity of the restrictions of the tangent map to the fibers defines an open subset of $C^r_{\text{\rm SF}}(X,\FF;X',\FF')$ for $1\le r\le\infty$.
			
			\item\label{i: proper} $\Prop^r(X,\FF;X',\FF')$ is open in $C^r_{\text{\rm SF}}(X,\FF;X',\FF')$ for $0\le r\le\infty$.
		
		\end{enumerate}
\end{prop}

\begin{proof}
	Adapt the proofs of \cite[Theorems~2.1.1 and~2.1.2]{Hirsch1976}.
\end{proof}

For general $C^r$ foliated maps $X\to X'$, $r\ge1$, the injectivity/surjectivity of the restrictions of their tangent maps to the fibers does not have any consequence on their transverse behavior, given by the induced morphisms $\HH\to\HH'$. Thus the foliated immersions/submersions or foliated local homeomorphisms cannot be described using only the tangent map. So conditions on the induced morphisms $\HH\to\HH'$ must be added to extend some deeper results. For this reason, we use a second version of weak/strong $C^r$ topology introduced in \cite{AlvarezMasa2006}, which is finer than the weak/strong foliated $C^r$ topology. The {\em strong plaquewise $C^r$ topology\/} has a base of open sets $\NN_{\text{\rm P}}^k(\phi,\UU,\UU',\KK,\EE)$, defined by adding the condition $p'_{a_i}\phi=p'_{a_i}\psi$ on every $K_i$ to the above definition of $\NN_{\text{\rm F}}^k(\phi,\UU,\UU',\KK,\EE)$; using~\eqref{xi'_a phi xi_i^-1}, this extra condition can be also written as $\phi^2_{a_ii}=\psi^2_{a_ii}$ on $p_i(K_i)$ for all $i$. The {\em weak plaquewise $C^r$ topology\/} is similarly defined by requiring the conditions only for finite families of indices $i$. The subindex ``WP/SP'' will be added to the notation to indicate that the weak/strong plaquewise $C^r$ topology is considered in a family of $C^r$ foliated maps. Note that, if two foliated maps are close enough in $C^r_{\text{\rm SP}}(X,\FF;X',\FF')$, then they induce the same morphism $\HH\to\HH'$; in fact, they are leafwisely homotopic if $r=\infty$, as follows by taking basic open sets $\NN_{\text{\rm P}}^k(\phi,\UU,\UU',\KK,\EE)$ as above where the plaques of the foliated charts in $\UU'$ are convex balls in the leaves for a given Riemannian metric on $X'$, and then using geodesic segments to define homotopies.

With the strong plaquewise $C^r$ topology, we can continue the direct extensions of results about spaces of $C^r$ maps between manifolds.  

\begin{prop}\label{p: list, SP}
	The following properties hold:
		\begin{enumerate}	
		
			\item\label{i: foliated embeddings} $\Emb^r(X,\FF;X',\FF')$ is open in $C^r_{\text{\rm SP}}(X,\FF;X',\FF')$ for $1\le r\le\infty$.
			
			\item\label{i: closed foliated embeddings} For $1\le r\le\infty$, the set of closed $C^r$ foliated embeddings is open in $C^r_{\text{\rm SP}}(X,\FF;X',\FF')$.

			\item\label{i: Diffeo^r is open} $\Diffeo^r(X,\FF;X',\FF')$ is open in $C^r_{\text{\rm SP}}(X,\FF;X',\FF')$ for $1\le r\le\infty$.
			
			\item\label{i: C^s is dense in C^r_SP} $C^s(X,\FF;X',\FF')$ is dense in $C^r_{\text{\rm SP}}(X,\FF;X',\FF')$ for $0\le r<s\le\infty$.
			
			\item\label{i: Diffeo^s is dense in Diffeo^r_SP} $\Diffeo^s(X,\FF;X',\FF')$ is dense in $\Diffeo^r_{\text{\rm SP}}(X,\FF;X',\FF')$ for $1\le r<s\le\infty$.
			
			\item\label{i: C^r => C^s} If $1\le r<\infty$, any $C^r$ foliated space is $C^r$ diffeomorphic to a $C^\infty$ foliated space.
			
			\item\label{i: C^s diffeomorphic <=> C^r diffeomorphic} If $1\le r<s\le\infty$, two $C^s$ foliated spaces are $C^s$ diffeomorphic if and only if they are $C^r$ diffeomorphic.
		
		\end{enumerate}
\end{prop}

\begin{proof}
	Adapt the proofs of \cite[Theorems~2.1.4,~2.1.6,~2.2.6,~2.2.7,~2.2.9 and~2.2.10, and Corollary~2.1.6]{Hirsch1976}.
\end{proof}

Like in the case of manifolds, it easily follows from Proposition~\ref{p: list, SP}-\eqref{i: C^s is dense in C^r_SP} that, for $0\le r<s\le\infty$, if there is a $C^r$ leafwise homotopy between $C^s$ foliated maps, then there is a $C^s$ leafwise homotopy between them.

The above openness statements are stronger with the strong foliated $C^r$ topology, whereas the denseness statements are stronger for the strong plaquewise $C^r$ topology. There is no version of Proposition~\ref{p: list, SP}-\eqref{i: foliated embeddings} with the strong foliated $C^r$ topology (for instance, consider the case of compact spaces foliated by points). However we can prove a weaker form of that statement by using certain subspaces $C^r_{\text{\rm SF}}(X,\FF;X',\FF')$ defined as follows. A foliated map $\phi:X\to X'$ is called a {\em transverse embedding\/} (respectively, {\em transverse equivalence\/}) if the induced morphism $\Phi:\HH\to\HH'$ is generated by embeddings (respectively, $\Phi$ is an isomorphism). Observe that $\FF'(\phi(X))=X'$ if $\phi$ is a transverse equivalence. A subset $\MM\subset C(X,\FF;X',\FF')$ of transverse embeddings (respectively, transverse equivalences) is called {\em uniform\/} if there are some foliated atlases, $\UU$ of $X$ and $\UU'$ of $X'$ like in Section~\ref{ss: foliated sps}, such that, for all $\phi\in\MM$, the maps $\phi^2_{ai}$ in~\eqref{xi'_a phi xi_i^-1} are embeddings (respectively, open embeddings). Note that, if these properties hold with $\UU$ and $\UU'$, then they hold with all finer atlases. For example, $\Emb(X,\FF;X',\FF')$ consists of uniform transverse embeddings, and $\Homeo(X,\FF;X',\FF')$ consists of uniform transverse equivalences.

\begin{prop}\label{p: Emb^r(FF,FF') cap MM is open in MM}
	For $1\le r\le\infty$, let $\MM\subset C^r_{\text{\rm SF}}(X,\FF;X',\FF')$ be a uniform subspace of transverse embeddings. Then $\Emb^r(X,\FF;X',\FF')\cap\MM$ is open in $\MM$.
\end{prop}

\begin{proof}
	It is enough to prove the case $r=1$. For any $\phi\in\Emb^1(X,\FF;X',\FF')\cap\MM$, consider a basic open set $\NN_1:=\NN_{\text{\rm F}}^1(\phi,\UU,\UU',\KK,\EE)$ in $C^1_{\text{\rm SF}}(X,\FF;X',\FF')$ as above. We can assume that $\KK$ (and therefore $\UU$) covers $X$, and $\UU'$ covers $X'$. After refinements, we can choose $\UU$, $\UU'$ and $\KK$ such that the maps $\psi^2_{ai}$ are embeddings for all $\psi\in\MM$, and the interiors $V_i:=\mathring{K_i}$ cover $X$. Take an open cover $\{W_i\}$ of $X$ with $\ol{W_i}\subset V_i$ for all $i$. By \cite[Lemma~1.3]{Hirsch1976}, we can choose $\EE$ such that the maps $\psi:p_i^{-1}(u)\cap V_i\to p_{a_i}^{\prime-1}(\psi^2_{a_ii}(u))$ are $C^1$ embeddings for $u\in p_i(V_i)$ and $\psi\in\NN_1$. Hence $\psi:V_i\to X'$ is a $C^1$ foliated embedding for all $\psi\in\NN_1\cap\MM$. 
	
	Now, we adapt the final part of the proof of \cite[Theorem~1.4]{Hirsch1976} as follows. Since $\phi$ is an embedding, we get disjoint open subsets $V'_i,W'_i\subset X'$ for every $i$ such that $\phi(\ol{W_i})\subset W'_i$ and $\phi(X\setminus V_i)\subset V'_i$. Then it is easy to find a neighborhood $\NN_0$ of $\phi$ in $C_{\text{\rm SF}}(X,\FF;X',\FF')$ so that $\psi(\ol{W_i})\subset W'_i$ and $\psi(X\setminus V_i)\subset V'_i$ for all $\psi\in\NN_0$. We finally obtain $\NN_0\cap\NN_1\cap\MM\subset\Emb^1(X,\FF;X',\FF')$.
\end{proof}

\begin{prop}\label{p: Diffeo^r(FF,FF') cap MM is open in MM}
	For $1\le r\le\infty$, let $\MM\subset C^r_{\text{\rm SF}}(X,\FF;X',\FF')$ be a uniform subspace of transverse equivalences. Then $\Diffeo^r(X,\FF;X',\FF')\cap\MM$ is open in $\MM$.
\end{prop}

\begin{proof}
	We adapt the proofs of \cite[Corollary~1.6 and Theorem~1.6]{Hirsch1976}. The set
		\[
			\MM'=\{\,\phi\in \Prop^r(X,\FF;X',\FF')\mid T_x\phi\ \text{is surjective}\ \forall x\in X\,\}
		\]
	is closed in $C^r_{\text{\rm SF}}(X,\FF;X',\FF')$ by Proposition~\ref{p: list, SF}-\eqref{i: T_x phi injective/surjective for all x},\eqref{i: proper}. On the other hand, $\Emb^r(X,\FF;X',\FF')\cap\MM$ is open in $\MM$ by Proposition~\ref{p: Emb^r(FF,FF') cap MM is open in MM}. Thus the result follows because $\Emb^r(X,\FF;X',\FF')\cap\MM'=\Diffeo^r(X,\FF;X',\FF')$.
\end{proof}

According to Proposition~\ref{p: list, SP}-\eqref{i: C^r => C^s},\eqref{i: C^s diffeomorphic <=> C^r diffeomorphic}, we will only consider either ($C^0$) foliated spaces or $C^\infty$ foliated spaces from now on.

\begin{prop}\label{p: T_x phi iso for all x}
	Let $\phi:X\to X'$ be a foliated map. Suppose that $X'$ is equipped with a $C^\infty$ structure. Then there is at most one $C^\infty$ structure on $X$ such that $\phi$ is $C^\infty$ and $T_x\phi$ is an isomorphism for all $x\in X$.
\end{prop}

\begin{proof}
	Consider two $C^\infty$ structures on $X$, and take $C^\infty$ foliated charts, $\xi_1:U_1\to B_1\times T_1$ of the first $C^\infty$ structure on $X$, $\xi_2:U_2\to B_2\times T_2$ of the second $C^\infty$ structure on $X$, and $\xi':U'\to B'\times T'$ of the $C^\infty$ structure on $X'$. We can assume that $U_2\subset U_1$ and $\phi(U_1)\subset U'$. Then
		\begin{align*}
			\xi'\phi\xi_1^{-1}(\bv_1,u_1)&=(g'_1(\bv_1,u_1),h'_1(u_1))\;,\\
			\xi'\phi\xi_2^{-1}(\bv_2,u_2)&=(g'_2(\bv_2,u_2),h'_2(u_2))\;,\\
			\xi_1\xi_2^{-1}(\bv_2,u_2)&=(g_{12}(\bv_2,u_2),h_{12}(u_2))\;,
		\end{align*}
	for $(\bv_k,u_k)\in B_k\times T_k$, $k=1,2$, where $g'_k:B_k\times T_k\to B'$ has partial derivatives of arbitrary order with respect to $\bv_k$, continuous on $B_k\times T_k$, and $g_{12}:B_2\times T_2\to B_1$ is continuous. Moreover the differential map of $g'_1$ with respect to $\bv_1$ is an isomorphism at any point. Therefore, by the inverse function theorem, we can assume that $g'_1(\cdot,u_1):B_1\to g'_1(B_1\times\{u_1\})$ is a $C^\infty$ diffeomorphism for all $u_1\in T_1$. Its inverse function is denoted by $\bar g'_1(\cdot,u_1):g'_1(B_1\times\{u_1\})\to B_1$. For any small ball $B'_0\subset B'$, let $T_{10}\subset T_1$ be the open subset that consists of the points $u_1\in T_1$ such that $B'_0\subset g'_1(B_1\times\{u_1\})$. It also follows from the inverse function theorem that the partial derivatives of arbitrary order of $\bar g'_1(\cdot,u_1):B'_0\to B_1$ depend continuously on $u_1$. Since
		\[
			g_{12}(\bv_2,u_2)=\bar g_1(g_2(\bv_2,u_2),h_{12}(u_2))
		\]
	on $B_2\times h_{21}(T_{10})$, the function $g_{12}:B_2\times h_{21}(T_{10})\to B_1$ has partial derivatives of arbitrary order with respect to $\bv_2$, continuous on $B_2\times h_{21}(T_{10})$.
\end{proof}

\subsection{Center of mass}\label{ss: center of mass}

In Section~\ref{s: C^infty local transverse action}, we will use the center of mass of a mass distribution on a Riemannian
manifold $M$ \cite{Karcher1977}, \cite[Section~IX.7]{Chavel2006}.

Let $\Omega\subset M$ be a compact submanifold with boundary with $\dim\Omega=\dim M$. For $0\le r\le\infty$, let $\CC(\Omega)$ be the set of functions $f\in
C^{r+2}(\Omega)$ such that $\grad f$ is an outward pointing
vector field on $\partial\Omega$ and $\Hess f$ is positive definite on
the interior $\mathring\Omega$ of $\Omega$. Note that $\CC(\Omega)$ is
open in the Banach space $C^{r+2}(\Omega)$ with the norm $\|\
\|_{C^{r+2},\Omega,g}$, and therefore it is a $C^\infty$ Banach
manifold. Moreover $\CC(\Omega)$ is preserved by the operations of sum
and product by positive numbers. Any $f\in\CC(\Omega)$ attains its
minimum value at a unique point $\bm_\Omega(f)\in\mathring\Omega$, defining a
function $\bm_\Omega:\CC(\Omega)\to\mathring\Omega$.

\begin{lem}[{\cite[Lemma~10.1 and Remark~11-(ii)]{AlvarezBarralCandel2016}}]\label{l: bm_Omega is C^r}
	The map $\bm_\Omega$ is $C^r$.
\end{lem}

Suppose that $M$ is connected and complete.  Let $(A,\mu)$ be a probability space, $B$ a convex open ball of radius $r>0$ in $M$, and $f:A\to B$ a measurable map, which is called a \emph{mass distribution} on $B$. Consider the $C^\infty$ function $P_{f,\mu}:B\to\R$ defined by
\[
P_{f,\mu}(x)=\frac{1}{2}\,\int_Ad(x,f(a))^2\,\mu(a)\;.
\]

\begin{prop}[{H.~Karcher  \cite[Theorem~1.2]{Karcher1977}}]\label{p: Karcher}
We have the following:
\begin{enumerate}

\item\label{i: grad P_f} $\grad P_{f,\mu}$ is an outward pointing vector 
field on the boundary $\partial\ol B$.

\item\label{i: Hess P_f} If $\delta>0$ is an upper bound for the 
sectional curvatures of $M$ in $B$, and $2 r<\pi/2\sqrt{\delta}$, then 
$\Hess P_{f,\mu}$ is positive definite on $B$.

\end{enumerate}
\end{prop}

If the hypotheses of Proposition~\ref{p: Karcher} are satisfied, then 
$P_{f,\mu}\in\CC(\ol B)$, and therefore $\CC_{f,\mu}:=\bm_{\ol B}(P_{f,\mu})\in B$ is defined and called the \emph{center of 
mass} of $f$ (with respect to $\mu$). This point is independent of the choice of $B$ satisfying the above conditions. The following is a consequence of Lemma~\ref{l: bm_Omega is C^r}.

\begin{cor}[{\cite[Corollary~10.3]{AlvarezBarralCandel2016}; cf.\ \cite[Corollary~1.6]{Karcher1977}}]\label{c: CC_f,mu}
        The following properties hold:
        \begin{enumerate}

                \item $\CC_{f,\mu}$ depends continuously on $f$ and the metric tensor of $M$.

                \item If $A$ is the Borel $\sigma$-algebra of a metric space, then $\CC_{f,\mu}$ depends continuously on $\mu$ in the weak-$*$ topology.
                
	\end{enumerate}
\end{cor}

Consider the following particular case. Let $N$ be a $C^\infty$ manifold, $\phi=(\phi_1,\dots,\phi_k):N\to M^k$ a $C^\infty$ map, and $\lambda=(\lambda_1,\dots,\lambda_k)$ a finite $C^\infty$ partition of unity of $N$. For every $x\in N$, consider the probability measure $\mu_{\phi,\lambda,x}=\sum_{i=1}^k\lambda_i(x)\,\delta_{\phi_i(x)}$, where $\delta_y$ denotes the Dirac mass at every $y\in M$. Suppose that, for all $x\in N$, the points $\phi_1(x),\dots,\phi_k(x)$ lie in a ball $B_x$ of $M$ satisfying the conditions of Proposition~\ref{p: Karcher}. Then we can define center of mass $\CC_{\phi,\lambda,x}$ of $\id_{B_x}$ with respect to $\mu_{\phi,\lambda,x}$, which is independent of the choice of $B_x$. The following sharpening of Corollary~\ref{c: CC_f,mu} also follows from Lemma~\ref{l: bm_Omega is C^r}.

\begin{cor}\label{c: CC_bar lambda,bar x}
        The map $N\to M$, $x\mapsto\CC_{\phi,\lambda,x}$, is $C^\infty$.
\end{cor}

\section{Molino's description}\label{s: Molino}

Consider the notation of Section~\ref{ss: foliated sps} in the rest of the paper. 

\begin{proof}[Proof of Theorem~\ref{mt: Molino}]
	Most of the properties stated in this theorem were already proved in \cite[Theorem~A]{AlvarezMoreira2016}. It only remains to prove the part concerning $H$.  For this purpose, we have to recall the construction of $G$, $\widehat X_0$, $\widehat\FF_0$ and $\hat\pi_0$. We can assume that $X$ satisfies the conditions of equicontinuity and strong quasi-analyticity with the same set $S$, and that $\ol\HH$ satisfies the conditions of equicontinuity and strong quasi-analyticity with the induced set $\ol S$. Let $\ol S_{\text{\rm c-o}}$ be the space $\ol S$ with the restriction of the compact-open topology on the set of partial maps $T\rightarrowtail T$ with open domain \cite{Abd-AllahBrown1980}. Consider the subspace
		\[
			\ol S_{\text{\rm c-o}}*T=\{\,(g,u)\in\ol S\times T\mid u\in\dom g\,\}\subset\ol S_{\text{\rm c-o}}\times T\;,
		\]
	and equip the set $\widehat T$ of all germs of maps in $\ol S$ (or $\ol\HH$) with the final topology induced by the germ map $\germ:\ol S_{\text{\rm c-o}}*T\to\widehat T$ (this is not the restriction of the sheaf topology). Consider the restrictions $s,t:\widehat T\to T$ of the source and target maps. The space $\widehat T$ is locally compact and Polish, and $\hat\pi:=(s,t):\widehat T\to T\times T$ is continuous and proper.
	
	Fix some point $u_0\in T_{i_0}\subset T$. Then the subspace $\widehat T_0:=s^{-1}(u_0)\subset\widehat T$ is locally compact and Polish. This definition is different from the one given in \cite[Section~3D]{AlvarezMoreira2016}, where $\widehat T_0=t^{-1}(u_0)$ was considered. This change can be made because the inversion of local transformations defines a homeomorphism of $\ol S_{\text{\rm c-o}}$ \cite[Proposition~3.1]{AlvarezMoreira2016}, and therefore the germ inversion defines a homeomorphism of $\widehat T$, which becomes a topological groupoid by \cite[Proposition~10]{Abd-AllahBrown1980}. The rest of definitions and arguments of \cite[Sections~3D--3G]{AlvarezMoreira2016} must be changed accordingly. For instance, take $\hat\pi_0=t:\widehat T_0\to T$ (instead of $\hat\pi_0=s$, used in \cite{AlvarezMoreira2016}), which is open, continuous and proper, and its fibers are homeomorphic to each other \cite[Section~3D]{AlvarezMoreira2016}. We have $\widehat T_0\equiv\bigsqcup_i\widehat T_{i,0}$, where $\widehat T_{i,0}=\hat\pi_0^{-1}(T_i)$.

	Note that $H:=\hat\pi_0^{-1}(u_0)=\hat\pi^{-1}(u_0,u_0)$ becomes a compact Polish group since $\widehat T$ is a topological groupoid. Moreover the germ product defines a continuous free right action of $H$ on $\widehat T_0$ whose orbits are clearly equal to the fibers of $\hat\pi_0:\widehat T_0\to T$. Thus this map induces a continuous bijection $\widehat T_0/H\to T$. In fact this bijection is a homeomorphism, as easily follows by using also that $H$ is compact, $\widehat T_0$ is locally compact, and $T$ is Hausdorff.

	For any $h\in\HH$, define $\hat h:\hat\pi_0^{-1}(\dom h)\rightarrow\hat\pi_0^{-1}(\im h)$ by $\hat h(\germ(g,u_0))=\germ(hg,u_0)$ for $g\in\ol S$ with $u_0\in\dom g$ and $g(u_0)\in\dom h$ (instead of $\hat h(\germ(g,u))=\germ(gh^{-1},h(u))$ for $u\in\dom g\cap\dom h$ with $g(u)=u_0$, used in \cite{AlvarezMoreira2016}). The maps $\hat h$ are local transformations of $\widehat T_0$ satisfying $h\hat\pi_0=\hat\pi_0\hat h$, $\widehat{\id_T}=\id_{\widehat T_0}$, $\widehat{hh'}=\hat h\widehat{h'}$ and ${\hat h}^{-1}=\widehat{h^{-1}}$ \cite[Sections~3E]{AlvarezMoreira2016}. Moreover it is easy to see that every $\hat h$ is $H$-equivariant (note that $\dom\hat h$ and $\im\hat h$ are $H$-invariant). Let $\widehat\HH_0$ be the pseudogroup on $\widehat T_0$ generated by $\widehat S_0=\{\,\hat h\mid h\in S\,\}$. There is a local group $G$ and some dense finitely generated sublocal group $\Gamma \subset  G$ such that $\widehat{\HH}_0$ is equivalent to the pseudogroup $\GG$ generated by the local action of $\Gamma$ on $G$ by local left translations \cite[Proposition~3.41]{AlvarezMoreira2016}---this was proved by checking that $\widehat\HH_0$ is compactly generated, equicontinuous and strongly locally free, and its closure is also strongly locally free, and then applying Proposition~\ref{p: list, G}-\eqref{i: G}. Furthermore $\hat\pi_0$ generates a morphism $\widehat\HH_0\to\HH$.
	
	Let $\check U_{i,0}=U_i \times\widehat T_{i,0}\times \{i\}\equiv U_i \times\widehat T_{i,0}$, equipped with the product topology, and consider the topological sum
		\[
  			\check X_0:=\bigsqcup_i (U_i \times\widehat T_{i,0})=\bigcup_i\check U_{i,0}\;, 
		\]
and the closed subspaces
		\[
  			\widetilde U_{i,0}:=\{\,(x,\gamma,i)\in\check U_{i,0} \mid p_i(x)=\hat\pi_0(\gamma)\,\}\subset\check U_{i,0}\;,\quad
			\widetilde X_0:=\bigcup_i\widetilde U_{i,0}\subset\check X_0\;.
		\]
	Note that $\widetilde X_0$ is the topological sum of the spaces $\widetilde U_{i,0}$. Consider the equivalence relation ``$\sim$'' on $\widetilde X_0$ defined by $(x,\gamma,i)\sim(y,\delta,j)$ if $x=y$ and $\gamma= \widehat{h_{ji}}(\delta)$. Let $\widehat X_0$ be the corresponding quotient space, let \(q:\widetilde X_0\rightarrow \widehat X_0\) be the quotient map, let $[x,\gamma,i]=q(x,\gamma,i)$, let $\widehat U_{i,0}=q(\widetilde U_{i,0})$, and let $\tilde{p}_{i,0}:\widetilde U_{i,0}\rightarrow\widehat T_{i,0}$ denote the restriction of $\check{p}_{i,0}:\check U_{i,0}\equiv U_i\times\widehat T_{i,0}\rightarrow\widehat T_{i,0}$, which induces a map $\hat p_{i,0}:\widehat U_{i,0}\rightarrow\widehat T_{i,0}$. Moreover a map $\hat\pi_0:\widehat X_0\rightarrow X$ is defined by $\hat\pi_0([x,\gamma,i])=x$. Observe that $\widehat U_{i,0}=\hat\pi_0^{-1}(U_i)$. Then $\widehat X_0$ is compact and Polish, $\{\widehat U_{i,0},\hat p_{i,0},\widehat{h_{ij}}\}$ is a defining cocycle of a minimal foliated structure $\widehat\FF_0$ on $\widehat X_0$, $\hat\pi_0$ is continuous and open, the fibers of $\hat\pi_0$ are homeomorphic to each other, and the restriction of $\hat\pi_0$ to the leaves of $\widehat X_0$ are the holonomy coverings of the leaves of $X$ \cite[Section~4B]{AlvarezMoreira2016}. In the proof of these properties, it was used that every restriction $q:\widetilde U_{i,0}\rightarrow \widehat U_{i,0}$ is a homeomorphism.
	
	Since every $\widehat T_{i,0}$ is $H$-invariant, we get an induced free right action of $H$ on every $\check U_{i,0}\equiv U_i \times\widehat T_{i,0}$, acting as the identity on the factor $U_i$, yielding a right $H$-action on $\check X_0$ by union. This restricts to a free right action of $H$ on $\widetilde X_0$, preserving every $\widetilde U_{i,0}$, because the $H$-orbits in $\widehat T_0$ are equal to the fibers $\hat\pi_0:\widehat T_0\to T$. Since moreover every $\widehat{h_{ij}}$ is $H$-equivariant, we get an induced right action on $\widehat X_0$, given by $[x,\gamma,i]\cdot\sigma=[x,\gamma\sigma,i]$ for $[x,\gamma,i]\in\widehat X_0$ and $\sigma\in H$. This action is also free because every restriction $q:\widetilde U_{i,0}\rightarrow \widehat U_{i,0}$ is a homeomorphism, and it is easy to see that its orbits equal the fibers of $\hat\pi_0:\widehat X_0\to X$. Finally note that every map $\hat p_{i,0}:\widehat U_{i,0}\rightarrow\widehat T_{i,0}$ is $H$-equivariant, and therefore $H$ acts on $\widehat X_0$ by foliated transformations.
\end{proof}

In the rest of this section, assume that $X$ satisfies the hypotheses of Theorem~\ref{mt: Molino}. Consider structures $(G,\Gamma,H,\widehat X_0,\hat\pi_0)$ satisfying the properties of its statement, considering $\widehat X_0$ as a foliated space and $H$-space, and $\Gamma$ is a finitely generated dense sublocal group $G$ so that the holonomy pseudogroup of $\widehat X_0$ is represented by the pseudogroup generated by the left local action of $\Gamma$ on $G$ by local left translations. 

It is said that two such structures, $(G,\Gamma,H,\widehat X_0,\hat\pi_0)$ and $(G',\Gamma',H',\widehat X'_0,\hat\pi'_0)$, are {\em equivalent\/} if there are a local isomorphism $\psi:G\rightarrowtail G'$ that restricts to a local isomorphism $\Gamma\rightarrowtail\Gamma'$, an isomorphism $\chi:H\to H'$, and a foliated $\chi$-equivariant homeomorphism $\phi:\widehat X_0\to\widehat X'_0$ such that $\hat\pi_0=\hat\pi'_0\phi$. In this case, $(\psi,\chi,\phi)$ is called an {\em equivalence\/}. This notion of equivalence is natural because it clearly means that the descriptions of the foliated space $X$ given by $(G,\Gamma,H,\widehat X_0,\hat\pi_0)$ and $(G',\Gamma',H',\widehat X'_0,\hat\pi'_0)$ are essentially the same, giving rise to equivalent invariants of $X$. For instance, $G$, $\Gamma$ and $H$ have the same algebraic and topological properties as $G'$, $\Gamma'$ and $H'$, and $\hat\pi_0$ is a principal bundle projection if and only if $\hat\pi'_0$ is also a principal bundle projection. 

The role of $\Gamma$ in the above structures is very important. Two foliated spaces satisfying the conditions of Theorem~\ref{mt: Molino} may have the same invariants $G$ and $H$, but different invariant $\Gamma$, and therefore their transverse dynamics may be quite different. In the foliated homogeneous case ($H=0$), this is well known for Lie foliations; for instance, all minimal Lie foliations of codimension one on tori have $G=\R$, but the rank of $\Gamma$ depends on the dimension of the tori. Homogeneous matchbox manifolds with the same $G$ and different $\Gamma$ can be constructed as suspensions (Section~\ref{ss: suspensions}), using the existence of non-isomorphic finitely presented residually finite groups with isomorphic profinite completions \cite{BridsonGrunewald2004,Nekrashevych2007}.

\begin{prop}[{Cf.~\cite[Propositions~3.43,~4.12 and~4.13]{AlvarezMoreira2016}}]\label{p: equivalent Molino}
	All structures $(G,\Gamma,H,\widehat X_0,\hat\pi_0)$ constructed in the proof of Theorem~\ref{mt: Molino} are equivalent.
\end{prop}

\begin{proof}
	We have to prove that the equivalence class of $(G,\Gamma,H,\widehat X_0,\hat\pi_0)$ is independent of the choices of $u_0$, $S$ and $\{U_i,p_i,h_{ij}\}$. Most of this is already proved in \cite[Propositions~3.43,~4.12 and~4.13]{AlvarezMoreira2016}. We only have to check what concerns $H$.
	
	To begin with, take another point of $u_1\in T_{i_1}\subset T$, and let $\widehat T_1$, $\hat\pi_1$, $\widehat S_1$, $\widehat{\HH}_1$, $G_1$, $\Gamma_1$ and $H_1$ be constructed like $\widehat T_0$, $\hat\pi_0$, $\widehat S_0$, $\widehat{\HH}_0$, $G_0:=G$, $\Gamma_0:=\Gamma$ and $H_0:=H$ by using $u_1$ instead of $u_0$. Now, for each $h\in\HH$, let us use the notation $\hat h_0:=\hat h\in\widehat\HH_0$, and let $\hat h_1:\hat\pi_1^{-1}(\dom h)\to\hat\pi_1^{-1}(\im h)$ be the map in $\widehat\HH_1$ defined like $\hat h$. In particular, the maps $(\widehat{h_{ij}})_1$ are defined like the maps $(\widehat{h_{ij}})_0:=\widehat{h_{ij}}$. There is some $f_0\in \overline{S}$ such that $u_0\in \dom f_0$ and $f_0(u_0)=u_1$. Let $\theta:\widehat T_0\rightarrow \widehat T_1$ be defined by $\theta(\germ(f,u_0))= \germ(ff_0^{-1},u_1)$ (instead of $\theta(\germ(f,x))= \germ(f_0f,x)$, like in \cite{AlvarezMoreira2016}). This map is a homeomorphism, and satisfies $\hat\pi_0=\hat\pi_1\theta$, $\dom \hat h_1= \theta(\dom \hat h_0)$ and $\hat h_1\theta=\theta\hat h_0$ for all $h\in S$, obtaining that $\theta$ generates an equivalence $\Theta:\widehat{\HH}_0\rightarrow \widehat{\HH}_1$ \cite[Proposition~3.42]{AlvarezMoreira2016}. For $k=0,1$, let $\GG_k$ be the pseudogroup on $G_k$ generated by local left translations by elements of $\Gamma_k$. Via equivalences $\widehat{\HH}_k\to\GG_k$, $\Theta$ corresponds to an equivalence $\Theta':\GG_0\to\GG_1$. Since the local right translations of $G_1$ generate equivalences of $\GG_1$, we can assume that the orbits of the identity elements correspond by the induced map $G_0/\GG_0\to G_1/\GG_1$. By Proposition~\ref{p: list, G}-\eqref{i: G >--> G'}, it follows that $\Theta'$ is generated by a local isomorphism $\psi:G_0\rightarrowtail G_1$ that restricts to a local isomorphism $\Gamma\rightarrowtail\Gamma'$. On the other hand, the conjugation mapping, $\germ(f,u_0)\mapsto\germ(f_0ff_0^{-1},u_1)$, defines an isomorphism $\chi:H_0\to H_1$ so that $\theta$ is $\chi$-equivariant.
	
	Now, define $\widehat X_1\equiv(\widehat X_1,\widehat\FF_1)$, $[x,\gamma,i]_1$ and $\hat\pi_1:\widehat X_1\to X$ like $\widehat X_0\equiv(\widehat X_0,\widehat\FF_0)$, $[x,\gamma,i]_0:=[x,\gamma,i]$ and $\hat\pi_0:\widehat X_0\to X$, using $\widehat T_1$, $\hat\pi_1:\widehat T_1\to T$ and the maps $(\widehat{h_{ij}})_1$ instead of $\widehat T_0$, $\hat\pi_0:\widehat T_0\to T$ and the maps $(\widehat{h_{ij}})_0$. According to \cite[Proposition~4.12]{AlvarezMoreira2016}, a foliated homeomorphism $\phi:\widehat X_0\to\widehat X_1$ is defined by $\phi([x,\gamma,i]_0)=[x,\theta(\gamma),i]_1$, which satisfies $\hat\pi_0=\hat\pi_1\phi$ and induces the equivalence $\Theta:\widehat\HH_0\to\widehat\HH_1$. Moreover $\phi$ is $\chi$-equivariant: for all $[x,\gamma,i]_0\in\widehat X_0$ and $\sigma\in H_0$,
		\begin{align*}
			\phi([x,\gamma,i]_0\cdot\sigma)&=\phi([x,\gamma\sigma,i]_0)=[x,\theta(\gamma\sigma),i]_1\\
			&=[x,\theta(\gamma)\chi(\sigma),i]_1=[x,\theta(\gamma),i]_1\cdot\chi(\sigma)\;.
		\end{align*}
		
	All choices of $S$ define the same space $\widehat T_0$ by \cite[Propositions~3.43]{AlvarezMoreira2016}, giving rise to the same Molino's description.
	
	To prove the independence of $\{U_i,p_i,h_{ij}\}$, it is enough to consider the case where $\{U_i,p_i,h_{ij}\}$ refines another defining cocycle $\{U'_a,p'_a,h'_{ab}\}$. Let $\HH'$ be the corresponding representative of the holonomy pseudogroup on $T'=\bigsqcup_aT'_a$. If $U_i\subset U'_{a_i}$, there is an induced open embedding $\phi_i:T_i\to T'_{a_i}$. These maps generate an equivalence $\Phi:\HH\to\HH'$. In fact,  $h'_{a_ia_j}\phi_j=\phi_ih_{ij}$. Let $u'_0=\phi_{i_0}(u_0)\in T'_{a_{i_0}}\subset T'$, and let $S'\subset\HH'$ be a generating subset such that $S^{\prime2}\subset S'=S^{\prime-1}$. We can also use $\{U'_a,p'_a,h'_{ab}\}$, $u'_0$ and $S'$ to define $\widehat T'_0$, $\hat\pi'_0:\widehat T'_0\to T'$ and $\widehat\HH'_0$ like $\widehat T_0$, $\hat\pi_0:\widehat T_0\to T$ and $\widehat\HH_0$; in particular, the generators $\widehat{h'_{ab}}$ of $\widehat\HH'_0$ are defined like the generators $\widehat{h_{ij}}$ of $\widehat\HH_0$.  We get open embeddings $\hat\phi_{i,0}:\widehat T_{i,0}\to\widehat T'_{a_i,0}$ defined by $\hat\phi_{i,0}(\germ(g,u_0))=\germ(\phi_ig\phi_{i_0}^{-1},u'_0)$, which generate an equivalence $\widehat\Phi_0:\widehat\HH_0\to\widehat\HH'_0$ (this is a corrected version of \cite[Proposition~3.44]{AlvarezMoreira2016}). Let $(G',\Gamma',H',\widehat X'_0,\hat\pi'_0)$ be the Molino's description defined with $\widehat T'_0$, $\hat\pi'_0:\widehat T'_0\to T'$ and the maps $\widehat{h'_{ab}}$. Let us use the notation $[x,\gamma',a]'$ for the element of $\widehat X'_0$ represented by a tern $(x,\gamma',a)$. Let $\GG$ and $\GG'$ be the pseudogroups on $G$ and $G'$ generated by the local left translations by elements of $\Gamma$ and $\Gamma'$. Via equivalences $\HH\to\GG$ and $\HH'\to\GG'$, $\widehat\Phi_0$ corresponds to an equivalence $\widehat\Phi'_0:\GG\to\GG'$. As above, we can assume that the orbits of the identity elements correspond by the induced map $G/\GG\to G'/\GG'$, and therefore, according to Proposition~\ref{p: list, G}-\eqref{i: G >--> G'}, $\widehat\Phi'_0$ is generated by a local isomorphism $\psi:G\rightarrowtail G'$ that restricts to a local isomorphism $\Gamma\rightarrowtail\Gamma'$. Moreover $\hat\phi_{i_0,0}$ restricts to an isomorphism $\chi:H\to H'$ so that any map in $\widehat\Phi_0$ is $\chi$-equivariant. Finally, a canonical foliated homeomorphism $\phi:\widehat X_0\to\widehat X'_0$ is well defined by $\phi([x,\gamma,i])=[x,\hat\phi_{i,0}(\gamma),a_i]'$ \cite[Proposition~4.13]{AlvarezMoreira2016}. It is easy to check that $\phi$ is $H$-equivariant. 
\end{proof}

By Proposition~\ref{p: equivalent Molino}, the equivalence class of any structure $(G,\Gamma,H,\widehat X_0,\hat\pi_0)$ constructed in the proof of Theorem~\ref{mt: Molino} can be called the {\em Molino's description\/} of $X$. According to the discussion of \cite[Section~1.E]{AlvarezMoreira2016}, these structures are kind of a topological interpretation of the original Molino's description in the case of a Riemannian foliation. That similarity can be indeed realized as an equivalence between the original Molino's description and ours in that case. According to Molino's terminology, the local isomorphism class of $G$ is called the {\em structural local group\/} \cite{AlvarezMoreira2016}, and, with the terminology of \cite{DyerHurderLukina2016,DyerHurderLukina2017}, $\widehat X_0$ will be called the {\em Molino space\/} and $H$ the {\em discriminant group\/}. 

\begin{prop}\label{p: G-foliated space <=> H = e}
	$X$ is a $G$-foliated space for some local group $G$ if and only if its discriminant group is trivial.
\end{prop}

\begin{proof}
	The ``if'' part of the statement is directly given by Theorem~\ref{mt: Molino}. To prove the ``only if'' part, assume $X$ is a $G$-foliated space for some local group $G$. Thus $\ol\HH$ is strongly locally free, obtaining that $H=\{e\}$ according to the definition of $H$ given in the proof of Theorem~\ref{mt: Molino}. 
\end{proof} 

For every $\hat x\in\widehat X_0$, let $\widehat L_{\hat x}$ denote the leaf of $\widehat X_0$ through $\hat x$, and consider the identity $\widetilde L_x^{\hol}\equiv\widehat L_{\hat x}$ given by Theorem~\ref{mt: Molino}. The following result is a direct consequence of the construction in the proof of Theorem~\ref{mt: Molino}.

\begin{prop}\label{p: Hol(L_x_0,x_0) subset H}
	Let $x_0\in p_{i_0}^{-1}(u_0)\subset U_{i_0}\subset X$. Given $\hat x_0\in\hat\pi_0^{-1}(x_0)$, we have
		\begin{equation}\label{Hol(L_x_0,x_0)}
			\Hol(L_{x_0},x_0)=\{\,\gamma\in H\mid\widehat L_{\hat x_0}\cdot\gamma=\widehat L_{\hat x_0}\,\}\;,
		\end{equation}
	and the map $\widetilde L_{x_0}^{\hol}\equiv\widehat L_{\hat x_0}\hookrightarrow\widehat X_0$ becomes equivariant with respect to the homomorphism $\Hol(L_{x_0},x_0)\hookrightarrow H$.
\end{prop}

According to the proof of Proposition~\ref{p: equivalent Molino}, it follows from Proposition~\ref{p: Hol(L_x_0,x_0) subset H} that, for all $x\in X$ and $\hat x\in\hat\pi_0^{-1}(x)$, there is an isomorphism
	\[
		\Hol(L_x,x)\cong\{\,\gamma\in H\mid\widehat L_{\hat x}\cdot\gamma=\widehat L_{\hat x}\,\}
	\]
	so that the map $\widetilde L_x^{\hol}\equiv\widehat L_{\hat x}\hookrightarrow\widehat X_0$ becomes equivariant with respect to the induced injective homomorphism $\Hol(L_x,x)\to H$. Nevertheless this isomorphism is not canonical in general.
	
With the notation of the proof of Theorem~\ref{mt: Molino}, consider a pseudogroup equivalence $\Phi:\widehat{\HH}_0\to\GG$. Since $\widehat{\HH}_0$ is locally equivariant with respect to the right action of $H$, there is a right local action of $H$ on $G$ so that $\Phi$ and $\GG$ are locally equivariant by Proposition~\ref{p: chi'}. By the density of $\Gamma$ in $G$, it easily follows that $H$ can be identified with a compact subgroup of $G$ acting on $G$ by local right translations. Moreover $\Phi$ induces an equivalence between $\HH$ and the pseudogroup $\GG/H$ on $G/H$ generated by the induced local left action of $\Gamma$ on $G/H$.

\begin{prop}\label{p: H has no non-trivial normal subgroups of G}
$H$ has no non-trivial normal subgroups of $G$.
\end{prop}

\begin{proof}
Let $K$ be a normal subgroup of $G$ contained in $H$. After taking the closure if necessary, we can assume that $K$ is compact. Then $G':=G/K$ becomes a locally compact local group with a finitely generated sublocal group $\Gamma':=\Gamma K/K$ and a compact subgroup $H':=H/K$. Let $\GG'$ the pseudogroup generated by the local left action of $\Gamma'$ on $G'$ by local left translations. The induced local left actions of $\Gamma$ and $\Gamma'$ on $G/H\equiv G'/H'$ generate the same pseudogroup $\GG/H\equiv\GG'/H'$, which is equivalent to $\HH$. But, according to the proof of Theorem~\ref{mt: Molino}, both $H$ and $H'$ can be canonically identified with the germs of maps in $\overline{\GG/H}\equiv\overline{\GG'/H'}$ whose source and target is any fixed element. So $H'\equiv H$, yielding $K=\{e\}$. 
\end{proof}

\section{Foliated homogeneous foliated spaces}\label{s: foliated homogeneous}

The foliated space $X$ is called {\em foliated homogeneous\/} when the canonical left action of $\Homeo(X,\FF)$ on $X$ is transitive. Similarly, if $X$ is $C^\infty$, it is called {\em $C^\infty$ foliated homogeneous\/} when the canonical left action of $\Diffeo(X,\FF)$ on $X$ is transitive. A priori, $C^\infty$ foliated homogeneity is stronger than foliated homogeneity, but we will see that indeed they are equivalent conditions for compact minimal $C^\infty$ foliated spaces (Section~\ref{s: C^infty G-fol sp => C^infty foliated homogeneous}).

Take any complete metric $d$ inducing the topology of $X$, and let $D$ be the induced complete metric on $\Homeo(X)$ defined by
	\[
		D(\phi,\psi)=\sup_{x\in X}d(\phi(x),\psi(x))+\sup_{x\in X}d(\phi^{-1}(x),\psi^{-1}(x))\;.
	\]
In this way, $\Homeo(X)$ becomes a completely metrizable topological group, and its canonical left action on $X$ is continuous. Moreover it is easy to check that $\Homeo(X,\FF)$ is closed in $\Homeo(X)$, and therefore $\Homeo(X,\FF)$ is also a completely metrizable topological group. 

Suppose that $X$ is compact. Then $D$ induces the compact-open topology on $\Homeo(X)$, as follows from \cite[Theorem~3]{Arens1946}, obtaining that $\Homeo(X)$ is also second countable. So $\Homeo(X)$ is a Polish group, and $\Homeo(X,\FF)$ a Polish subgroup. Therefore, by a theorem of Effros \cite{Effros1965,vanMill2004}, if $X$ is foliated homogeneous, then the canonical left action of $\Homeo(X,\FF)$ on $X$ is {\em micro-transitive\/}; i.e., for all $x\in X$ and any neighborhood $\NN$ of $\id_X$ in $\Homeo(X,\FF)$, the set $\NN\cdot x$ is a neighborhood of $x$ in $X$.

\begin{proof}[Proof of Theorem~\ref{mt: foliated homogeneous => G-foliated space}]
	Clark and Hurder have proved that any $C^\infty$ homogeneous matchbox manifold is equicontinuous \cite[Theorem~5.2]{ClarkHurder2013}. Indeed, their argument applies to any compact minimal foliated homogeneous foliated space. Moreover the $C^\infty$ structure is not used in that result. Thus the conditions of our statement are enough to get that $(X,\FF)$ is equicontinuous.
	
	The rest of the proof uses the same main tool as in \cite[Theorem~5.2]{ClarkHurder2013}, the indicated theorem of Effros. 
	
	Let us prove that $\HH$ is strongly locally free. Recall that this means that there is some generating set $S$ with  $S^2\subset S=S^{-1}$ such that, for every $h\in S$, if $h$ is the identity in some non-empty open set, then it is the identity in its whole domain. Since $\{U_i\}$ is finite, there is some $\epsilon>0$ such that $d(\ol{U_i},X\setminus\widetilde U_i)<\epsilon$ for all $i$. Since the action of $\Homeo(X,\FF)$ on $X$ is micro-transitive, there is some $\delta>0$ such that, for all $x,y\in X$ with $d(x,y)<\delta$, there exists some $\phi\in\Homeo(X,\FF)$ so that $D(\phi,\id_X)<\epsilon$ and $\phi(x)=y$. 
	
	Since every $T_i$ has compact closure in $\widetilde T_i$, we easily get a finite open cover $\{T_{ia}\}$ of $T_i$ such that the $d$-diameter of every $\sigma_i(T_{ia})$ is smaller than $\delta$. Let $U_{ia}=\xi_i^{-1}(B_i\times T_{ia})$, $\xi_{ia}=\xi_i|_{U_{ia}}$, $\widetilde U_{ia}=\widetilde U_i$ and $\tilde\xi_{ia}=\tilde\xi_i$. By using $\{U_{ia},\xi_{ia}\}$ and $\{\widetilde U_{ia},\tilde\xi_{ia}\}$, varying $i$ and $a$, instead of $\{U_i,\xi_i\}$ and $\{\widetilde U_i,\tilde\xi_i\}$, it follows that we can assume that the $d$-diameter of every $\sigma_i(T_i)$ is smaller than $\delta$.
	
	Take $S$ equal to the family of the maps $h_\II$ for admissible sequences $\mathcal I$. Suppose that some $h_\II\in S$ fixes a point $u\in\dom h_\II$. Thus $\mathcal I=(i_0,\dots,i_\alpha)$ with $i_\alpha=i_0$. Let $x=\sigma_{i_0}(u)\in U_{i_0}$ and let $c:I\to X$ be a leafwise loop in $L_x$ based at $x$ and $\UU$-covered by $\II$. Take any point $v\in\dom h_\II$, and let $y=\sigma_{i_0}(v)\in U_{i_0}$. Since the $d$-diameter of $\sigma_{i_0}(T_{i_0})$ is smaller than $\delta$, according to our application of the Effros theorem, there is some $\phi\in\Homeo(X,\FF)$ with $\phi(x)=y$ and $d(c(t),\phi c(t))<\epsilon$ for all $t\in I$. Hence the leafwise path $\phi c:I\to X$ is $\widetilde\UU$-covered by $\II$. It follows that $\tilde h_\II(v)=p_{i_0}\phi c(1)=p_{i_0}\phi(x)=p_{i_0}(y)=v$, obtaining $h_\II(v)=v$. This shows that $h_\II=\id_{\dom h_\II}$, and therefore $\HH$ satisfies the condition of being strongly locally free with this $S$.
	
	$\HH$ is strongly quasi-analytic because it is strongly locally free, and therefore the hypotheses of Theorem~\ref{mt: Molino} are satisfied. In particular, the closure $\ol\HH$ is defined and generated by the set $\ol S$ induced by the above $S$. 
	
	Now, let us sharpen the above argument to prove that $\ol\HH$ is also strongly locally free, and therefore $(X,\FF)$ is a $G$-foliated space for some local group $G$ by Proposition~\ref{p: list, G}-\eqref{i: G}. For any $g\in\ol S$ with $O=\dom g$, there is a sequence of admissible sequences, $\mathcal I_k=(i_{k,0},\dots,i_{k,\alpha_k})$, such that $O\subset\dom h_{\mathcal I_k}$ for all $k$ and $g=\lim_kh_{\mathcal I_k}|_O$ in the compact-open topology. Thus $i_0:=i_{k,0}$ is independent of $k$. Suppose that $g(u)=u$ for some $u\in O$, which means that $u'_k:=h_{\mathcal I_k}(u)\to u$ as $k\to\infty$. So we can assume that $i_{k,\alpha_k}=i_0$ for all $k$. Let $x=\sigma_{i_0}(u)\in U_{i_0}$ and $x'_k=\sigma_{i_0}(u'_k)\in U_{i_0}$. We get $x'_k=\sigma_{i_0}(u'_k)\to\sigma_{i_0}(u)=x$ because $u'_k\to u$. For every $k$, there exists a leafwise path $c_k$, $\UU$-covered by $\II_k$, with $c_k(0)=x$ and $c_k(1)=x'_k$. For any $v\in O$, we have $v'_k:=h_{\II_k}(v)\to g(v)$, and let $y=\sigma_{i_0}(v)\in U_{i_0}$. As before, there is some $\phi\in\Homeo(X,\FF)$ such that $\phi(x)=y$ and $d(c_k(t),\phi c_k(t))<\epsilon$ for all $t\in I$, and let $y'_k:=\phi c_k(1)=\phi(x'_k)$. Hence the leafwise path $\phi c_k$ is $\widetilde\UU$-covered by $\II_k$, obtaining $p_{i_0}(y'_k)=\tilde h_{\II_k}(v)=h_{\II_k}(v)=v'_k$. Thus $v'_k\to v$ because $y'_k=\phi(x'_k)\to\phi(x)=y$ and $p_{i_0}(y)=v$. So $g(v)=v$, showing $g=\id_O$. Therefore $\ol\HH$ satisfies the condition of being strongly locally free with $\ol S$.
\end{proof}

\section{$C^\infty$ Molino's description}\label{s: C^infty Molino}

In this section, suppose that $X$ is $C^\infty$ and satisfies the hypotheses of Theorem~\ref{mt: Molino}, and let $(G,\Gamma,H,\widehat X_0,\hat\pi_0)$ represent its Molino's description. Recall that $H$ is a compact group acting on $\widehat X_0$ and $\pi_0:\widehat X_0\to X$ induces a homeomorphism $\widehat X_0/H \to X$.

\begin{prop}\label{p: C^infty Molino}
	$\widehat X_0$ has a unique $C^\infty$ structure so that $\hat\pi_0$ is $C^\infty$ and $T\hat\pi_0:T\widehat\FF_0\to T\FF$ restricts to isomorphisms between the fibers. Moreover the foliated $H$-action is also $C^\infty$.
\end{prop}

\begin{proof}
	If the foliated atlas $\UU$ defines a $C^\infty$ foliated structure on $X$, it is easy to check that the foliated atlas of $\widehat X_0$ constructed in the proof of Theorem~\ref{mt: Molino} (Section~\ref{s: Molino}) defines a $C^\infty$ foliated structure satisfying the stated properties.
	
	By Proposition~\ref{p: T_x phi iso for all x} applied to $\hat\pi_0$, the $C^\infty$ structure on $\widehat X_0$ is determined by the condition that $\hat\pi_0$ is $C^\infty$ and $T\hat\pi_0$ restricts to isomorphisms between the fibers.
\end{proof}

If $\widehat X_0$ is equipped with the unique $C^\infty$ structure given by Proposition~\ref{p: C^infty Molino}, then $(G,\Gamma,H,\widehat X_0,\hat\pi_0)$ is called the {\em $C^\infty$ Molino's description\/} of $X$.

\section{Right local transverse actions}\label{s: local transverse action}

\subsection{Topological right local transverse actions}\label{s: top local transverse action}

The foliated homeomorphisms leafwisely homotopic to the identity form a normal subgroup $\Homeo_0(X,\FF)$ of $\Homeo(X,\FF)$, obtaining the (possibly non-Hausdorff) topological group 
	\[
		\ol{\Homeo}(X,\FF)=\Homeo(X,\FF)/\Homeo_0(X,\FF)\;.
	\]

Suppose that $X$ is compact for the sake of simplicity. Then a {\em right local transverse action\/} of a local group $G$ on $X$ can be defined as a map $\phi:X\times O\to X$, for some $O\in\NN(G,e)$, such that $\phi^g:=\phi(\cdot,g)\in\Homeo(X,\FF)$ for all $g\in O$, and $O\to\ol{\Homeo}(X,\FF)$, $g\mapsto[\phi^g]$, is a local anti-homomorphism of $G$ to $\ol{\Homeo}(X,\FF)$. Two right local transverse actions, $\phi:X\times O\to X$ and $\psi:X\times P\to X$, are declared to be {\em equivalent\/} if there is some $Q\in\NN(G,e)$ such that $Q\subset O\cap P$ and the restrictions $\phi,\psi:X\times Q\to X$ are leafwise homotopic with respect to the foliated structure on $X\times Q$ with leaves $L\times\{g\}$, for leaves $L$ of $X$ and points $g\in Q$.

\begin{lem}\label{l: G is locally contractible}
	If $G$ is locally contractible, then the equivalence class of $\phi$ is determined by the induced local anti-homomorphism of $G$ to $\ol{\Homeo}(X,\FF)$.
\end{lem}

\begin{proof}
	Let $\psi:X\times P\to X$ be another right local transverse action inducing the same local anti-homomorphism of $G$ to $\ol{\Homeo}(X,\FF)$ as $\phi$. Thus there is some $Q\in\NN(G,e)$ such that $Q\subset O\cap P$ and $\phi^g$ is leafwisely homotopic to $\psi^g$ for all $g\in Q$. Since $G$ is locally contractible, we can suppose that there is a homotopy $E:Q\times I\to Q$ of the constant map $\const_{g_0}$ to $\id_Q$, for some point $g_0\in Q$. By choosing $Q$ small enough and using $g_0^{-1}Q$ instead of $Q$, we can also assume that $g_0=e$. Let $g_t=E(g,t)$ for $g\in Q$ and $t\in I$. Given any leafwise homotopy $H:X\times I\to X$ of $\phi^e$ to $\psi^e$, the map $F:X\times Q\times I\to X$, defined by
		\[
			F(x,g,t)=\psi^{g_t}(\psi^e)^{-1}\phi^{g_t^{-1}g}(\phi^e)^{-1}H(x,t),
		\] 
	is a leafwise homotopy between the restrictions $\phi,\psi:X\times Q\to X$, where $X\times Q$ is foliated with leaves $L\times\{g\}$, for leaves $L$ of $X$ and points $g\in Q$. This follows by using that $(\psi^e)^{-1}$, $(\phi^e)^{-1}$ and $H(\cdot,t)$ are leafwisely homotopic to $\id_X$, and $\psi^{g_t}$ and $\phi^{g_t}\phi^{g_t^{-1}g}$ are leafwisely homotopic to $\phi^{g_t}$ and $\phi^g$, respectively. 
\end{proof}

According to Lemma~\ref{l: G is locally contractible}, when $G$ is locally contractible, a right local transverse action of $G$ on $X$ could be defined as a local anti-homomorphism $G$ to $\ol{\Homeo}(X,\FF)$, given by a map $O\to\ol{\Homeo}(X,\FF)$, $g\mapsto[\phi^g]$, for some $O\in\NN(G,e)$ and some foliated map $\phi:X\times O\to X$ with $\phi^g\in\Homeo(X,\FF)$ for all $g\in O$, where $X\times O$ is foliated with leaves $L\times\{g\}$, for leaves $L$ of $X$ and points $g\in O$. This corresponds to the definition of right transverse action of Lie groups on foliated manifolds given in \cite{AlvarezKordyukov2008a}. But it seems impossible to extend Lemma~\ref{l: G is locally contractible} to arbitrary local groups, which motivates our more involved definition.

\begin{lem}\label{l: phi^e = id_X}
	We can assume $\phi^e=\id_X$.
\end{lem}

\begin{proof}
	Consider the foliated structure on $X\times O$ with leaves $L\times\{g\}$, for leaves $L$ of $X$ and points $g\in O$. The foliated map $\psi:X\times O\to X$, defined by $\psi^g:=\phi^g(\phi^e)^{-1}$, satisfies the stated conditions. In fact, if $H:X\times I\to X$ is a leafwise homotopy of $(\phi^e)^{-1}$ to $\id_X$, then $F:X\times O\times I\to X$, defined by $F(\cdot,g,t)=\phi^gH(\cdot,t)$, is a leafwise homotopy of $\phi$ to $\psi$.
\end{proof}

From now on, suppose that $\phi^e=\id_X$ according to Lemma~\ref{l: phi^e = id_X}. Then, since $X$ is compact, there is some $O'\in\NN(G,e)$ such that $O'\subset O$ and $\phi(U_i\times O')\subset\widetilde U_i$ for all $i$. The foliated restrictions $\phi:U_i\times O'\to\widetilde U_i$ induce maps $\bar\phi:T_i\times O'\to\widetilde T_i$, and let $\bar\phi:T\times O'\to\widetilde T$ denote their union. Then the restriction $\bar\phi:\Omega:=\bar\phi^{-1}(T)\to T$ is a right local action of $G$ on $T$, which will be said to be {\em induced\/} by $\phi$.

\begin{lem}\label{l: HH is locally equivariant}
	$\HH$ is locally equivariant (with respect to $\bar\phi:\Omega\to T$).
\end{lem}

\begin{proof}
	 It is enough to prove that the maps $h_{ij}$ are locally equivariant. Let $u\in p_j(U_i\cap U_j)$ and $g\in O'$, and take any $x\in U_i\cap U_j$ such that $p_j(x)=u$. We have $h_{ij}(u)=p_i(x)$, $\phi(x,g)\in\widetilde U_i\cap\widetilde U_j$ and $\bar\phi(u,g)=p_j\phi(x,g)$, yielding
	 	\[
			\tilde h_{ij}\bar\phi(u,g)=p_i\phi(x,g)=\bar\phi(p_i(x),g)=\bar\phi(h_{ij}(u),g)\;.
		\]
	So $h_{ij}\bar\phi(u,g)=\bar\phi(h_{ij}(u),g)$ for all $(u,g)$ in $(p_j(U_i\cap U_j)\times O')\cap\bar\phi^{-1}(T_i)$, which is an open neighborhood of $p_j(U_i\cap U_j)\times\{e\}$ in $\Omega$.
\end{proof}

\begin{lem}\label{l: [phi] determines [bar phi]}
	If $X$ has no holonomy, then the equivalence class of $\phi$ determines the equivalence class of $\bar\phi:\Omega\to T$.
\end{lem}

\begin{proof}
	Suppose that $\phi$ is equivalent to another right transverse local action $\psi:X\times P\to X$ with $\psi^e=\id_X$. Take some $P'\in\NN(G,e)$ such that $P'\subset P$ and $\psi(U_i\times P')\subset\widetilde U_i$ for all $i$. As above, consider the map $\bar\psi:T\times P'\to\widetilde T$ induced by the foliated restrictions $\psi:U_i\times P'\to\widetilde U_i$, whose restriction $\bar\psi:\Sigma:=\bar\psi^{-1}(T)\to T$ is a right local action of $G$ on $T$. For some $Q\in\NN(G,e)$ with $Q\subset O'\cap P'$, there is a leafwise homotopy $H:X\times Q\times I\to X$ between the foliated restrictions $\phi,\psi:X\times Q\to X$, where $X\times Q$ is foliated as before.
	
	\begin{claim}\label{cl: tilde p_i phi^g(x) = tilde p_i psi^g(x)}
		We have $\bar\phi=\bar\psi$ on $T\times Q'$ for some $Q'\in\NN(G,e)$ with $Q'\subset Q$.
	\end{claim}
	
	By absurdity, suppose that this assertion is not true. Then $\tilde p_{i_k}\phi^{g_k}(x_k)\ne\tilde p_{i_k}\psi^{g_k}(x_k)$ for some sequences, of indices $i_k$, of points $x_k\in U_{i_k}$, and $g_k\to e$ in $G$. Since $X$ is compact, we can assume that $i_k=i$ for all $k$, and $x_k\to x$ in $X$; thus $x\in\ol{U_i}\subset\widetilde U_i$. Consider the leafwise paths $c_k=H(x_k,g_k,\cdot)$ and $c=H(x,e,\cdot)$. Note that $c_k\to c$ in the compact-open topology, and $c$ is a loop in $L_x$ based at $x$ because $\phi^e=\psi^e=\id_X$. Let $\JJ=(j_0,\dots,j_\alpha)$ be an admissible sequence $\widetilde\UU$-covering $c$ with $j_0=j_\alpha=i$. Hence $\JJ$ also $\widetilde\UU$-covers $c_k$ for $k$ large enough, obtaining that $\tilde p_i\phi^{g_k}(x_k)\in\dom\tilde h_\JJ$ and $\tilde h_\JJ\tilde p_i\phi^{g_k}(x_k)=\tilde p_i\psi^{g_k}(x_k)$ for $k$ large enough. Since $\tilde p_i\phi^{g_k}(x_k)\to\tilde p_i(x)$ in $T_i$ and $\tilde h_\JJ$ is the identity on some neighborhood of $\tilde p_i(x)$ because $X$ has no holonomy, it follows that $\tilde h_\JJ\tilde p_i\phi^{g_k}(x_k)=\tilde p_i\phi^{g_k}(x_k)$ for $k$ large enough, yielding $\tilde p_i\phi^{g_k}(x_k)=\tilde p_i\psi^{g_k}(x_k)$ for $k$ large enough, a contradiction.
	
	By Claim~\ref{cl: tilde p_i phi^g(x) = tilde p_i psi^g(x)}, we get $\bar\phi=\bar\psi$ on $\Omega\cap\Sigma\cap(T\times Q')$, showing that the right local actions $\bar\phi:\Omega\to T$ and $\bar\psi:\Sigma\to T$ are equivalent.
\end{proof}

\subsection{$C^\infty$ right local transverse actions}\label{s: C^infty local transverse action}

From now on, assume that $X$ is $C^\infty$, and consider also the (possibly non-Hausdorff) topological group 
	\[
		\ol{\Diffeo}(X,\FF)=\Diffeo(X,\FF)/\Diffeo_0(X,\FF)\;,
	\]
where $\Diffeo_0(X,\FF)$ is the normal subgroup of $\Diffeo(X,\FF)$ consisting of the foliated diffeomorphisms that are leafwisely homotopic to $\id_X$; i.e., $\Diffeo_0(X,\FF)=\Diffeo(X,\FF)\cap\Homeo_0(X,\FF)$. It is said that the right local transverse action $\phi:X\times O\to X$ is $C^\infty$ if it is $C^\infty$ as foliated map, where $X\times O$ is foliated with leaves $L\times\{g\}$, for leaves $L$ of $X$ and points $g\in G$, and moreover $\phi^g\in\Diffeo(X,\FF)$ for all $g\in O$, and $O\to\ol{\Diffeo}(X,\FF)$, $g\mapsto[\phi^g]$, is a local anti-homomorphism of $G$ to $\ol{\Diffeo}(X,\FF)$. A {\em $C^\infty$ equivalence\/} between two $C^\infty$ right local transverse actions is defined like in the case of right local transverse actions. Suppose also that $\phi$ is $C^\infty$ from now on, and consider the induced right local action $\bar\phi:\Omega\to T$ defined in Section~\ref{s: top local transverse action}.

\begin{lem}\label{l: [bar phi] determines [phi]}
	The $C^\infty$ equivalence class of $\phi$ is determined by the equivalence class of $\bar\phi:\Omega\to T$.
\end{lem}

\begin{proof}
	Let $\psi:X\times P\to X$ be another $C^\infty$ right local transverse action of $G$ on $X$ with $\psi^e=\id_X$. Take some $P'\in\NN(G,e)$ such that $P'\subset P$ and $\psi(U_i\times P')\subset\widetilde U_i$ for all $i$. Like in Section~\ref{s: top local transverse action}, let $\bar\psi:T\times P'\to\widetilde T$ be induced by the foliated restrictions $\psi:U_i\times P'\to\widetilde U_i$, and consider the right local action $\bar\psi:\Sigma:=\bar\psi^{-1}(T)\to T$. Suppose that $\bar\psi=\bar\phi$ on some open neighborhood $\Theta$ of $T\times\{e\}$ in $\Omega\cap\Sigma$. So $\tilde p_i\phi(x,g)=\tilde p_i\psi(x,g)$ for all $i$ and $(x,g)\in U_i\times(O\cap P)$ with $(p_i(x),g)\in\Theta$. Since $X$ is compact, the open neighborhood of $X\times\{e\}$ in $X\times(O\cap P)$,
		\[
			\bigcup_i\{\,(x,g)\in U_i\times(O\cap P)\mid(p_i(x),g)\in\Theta\,\}\;,
		\]
	contains $X\times Q$ for some $Q\in\NN(G,e)$. Hence $\phi(x,g)$ and $\psi(x,g)$ lie in the same plaque of some $\widetilde U_i$ for all $(x,g)\in X\times Q$. We can further assume that the plaques of the foliated charts in $\widetilde\UU$ are convex for some choice of a Riemannian metric on $X$, obtaining a $C^\infty$ leafwise homotopy between the foliated restrictions $\phi,\psi:X\times Q\to X$ by using geodesic segments in the leaves, where $X\times Q$ is foliated with leaves $L\times\{g\}$, for leaves $L$ of $X$ and points $g\in Q$. Therefore $\phi$ and $\psi$ are $C^\infty$ equivalent.
\end{proof}

\begin{prop}\label{p: exists a right local transverse action}
	 If $X$ is without holonomy, then the assignment of the induced right local action defines a bijection of the set of $C^\infty$ equivalence classes of $C^\infty$ right local transverse actions of $G$ on $X$ to the set of equivalence classes of right local actions of $G$ on $T$ satisfying that $\HH$ is locally equivariant.
\end{prop}

\begin{proof}
	By Lemmas~\ref{l: HH is locally equivariant},~\ref{l: [phi] determines [bar phi]} and~\ref{l: [bar phi] determines [phi]}, it only remains to prove that, if $\HH$ is locally equivariant with respect to a right local action $\chi:\Sigma\to T$ of $G$ on $T$, then $\chi$ is induced by some $C^\infty$ right local transverse action of $G$ on $X$.
	
	By Proposition~\ref{p: chi'},  $\widetilde\HH$ is locally equivariant with respect to some right local action $\tilde\chi:\widetilde\Sigma\to\widetilde T$ of $G$ on $\widetilde T$, whose restriction to $T$ is equivalent to $\chi$. Since $T$ is relatively compact in $\widetilde T$, there is some $P\in\NN(G,e)$ such that $P\subset O$, $\ol T\times P\subset\widetilde\Sigma$ and $\tilde\chi(\ol{T_i}\times P)\subset\widetilde T_i$ for all $i$. Then, for $x\in\ol{U_i}\subset\widetilde U_i$ with $\tilde\xi_i(x)=(\bv,u)$ and $g\in P$, the point $\phi_i(x,g):=\tilde\xi_i^{-1}(\bv,\tilde\chi(u,g))\in\widetilde U_i$ is well defined because $u=\tilde p_i(x)\in\tilde p_i(\ol{U_i})=\ol{T_i}$. 
	
	\begin{claim}\label{cl: tilde p_i phi_i(x,g) = tilde p_i phi_j(x,g)}
		There is some $Q\in\NN(G,e)$ such that $Q\subset P$ and, if $x\in U_i\cap U_j$ and $g\in Q$, then $\phi_i(x,g),\phi_j(x,g)\in\widetilde U_i$ and $\tilde p_i\phi_i(x,g)=\tilde p_i\phi_j(x,g)$.
	\end{claim}
	
	By absurdity, suppose that this assertion is not true. So $\tilde p_{i_k}\phi_{i_k}(x_k,g_k)\ne\tilde p_{i_k}\phi_{j_k}(x_k,g_k)$ for some sequences, of indices $i_k,j_k$, of points $x_k\in U_{i_k}\cap U_{j_k}$, and $g_k\to e$ in $P$. Since $X$ is compact, we can assume that $i_k=i$ and $j_k=j$ for all $k$, and $x_k\to x$ in $X$. Thus $x\in\ol{U_i}\cap\ol{U_j}\subset\widetilde U_i\cap\widetilde U_j$, $\phi_i(x,g_k)\in\widetilde U_i$ and $\phi_j(x,g_k)\in\widetilde U_j$. Let $u_k=\tilde p_j(x_k)$ and $u=\tilde p_j(x)$. Since $\widetilde\HH$ is locally equivariant, there are some open neighborhood $W$ of $u$ in $\dom\tilde h_{ij}$ and $Q\in\NN(G,e)$ such that $Q\subset P$, $W\times Q,\tilde h_{ij}(W)\times Q\subset\widetilde\Sigma$, $\tilde\chi(W\times Q)\subset\dom\tilde h_{ij}$ and $\tilde\chi(\tilde h_{ij}(w),g)=\tilde h_{ij}\tilde\chi(w,g)$ for all $(w,g)\in W\times Q$. Take some open neighborhood $N$ of $x$ in $X$ so that $\ol N\subset\widetilde U_i\cap\widetilde U_j$ and $\tilde p_j(\ol N)\subset W$. We can choose $Q$ such that $\phi_j(\ol N\times Q)\subset\widetilde U_i\cap\widetilde U_j$, and therefore
		\[
			\tilde p_j\phi_j\big(\ol N\times Q\big)=\tilde\chi\big(\tilde p_j\big(\ol N\big)\times Q\big)\subset\dom\tilde h_{ij}\;.
		\]
	For $k$ large enough, we have $(x_k,g_k)\in N\times Q$, obtaining
		\[
			\tilde p_i\phi_i(x_k,g_k)=\tilde\chi(\tilde h_{ij}(u_k),g_k)=\tilde h_{ij}\tilde\chi(u_k,g_k)=\tilde p_i\phi_j(x_k,g_k)\;,
		\]
	a contradiction that proves Claim~\ref{cl: tilde p_i phi_i(x,g) = tilde p_i phi_j(x,g)}.
	
	Given any Riemannian metric on $X$, we can assume that the plaques of every $(U_i,\xi_i)$ and $(\widetilde U_i,\tilde\xi_i)$ are convex balls of diameter $<\pi/2\sqrt{\delta}$, where $\delta>0$ is an upper bound for the sectional curvature of the leaves 
	
	Consider the open neighborhood $Q$ of $e$ in $P$ given by Claim~\ref{cl: tilde p_i phi_i(x,g) = tilde p_i phi_j(x,g)}, and let $\{\lambda_i\}$ be a $C^\infty$ partition of unity of $X$ subordinated to $\{U_i\}$. For all $(x,g)\in X\times Q$, a probability measure on $X$ is well defined by $\mu_{x,g}=\sum_i\lambda_i(x)\,\delta_{\phi_i(x,g)}$, where $\delta_y$ denotes the Dirac mass at every $y\in X$. By Claim~\ref{cl: tilde p_i phi_i(x,g) = tilde p_i phi_j(x,g)}, if $x\in\supp\lambda_i$, then $\mu_{x,g}$ is supported in the plaque $\tilde p_i^{-1}(\chi(p_i(x),g))$ of $(\widetilde U_i,\tilde\xi_i)$. Then, by Corollary~\ref{c: CC_bar lambda,bar x}, a $C^\infty$ foliated map $\phi:X\times Q\to X$ is defined by taking $\phi(x,g)$ equal to the center of mass of $\mu_{x,g}$ in the common leaf through the points $\phi_i(x,g)$, where $X\times Q$ is foliated with leaves $L\times\{g\}$, for leaves $L$ of $X$ and points $g\in Q$. Let $\phi^g=\phi(\cdot,g):X\to X$ for $g\in Q$. Note that $\phi^g(U_i)\subset\widetilde U_i$, and $\phi^e=\id_X$ because $\phi_i(x,e)=x$ for $x\in\ol{U_i}$.
	
	 \begin{claim}\label{cl: phi^gh = phi^h phi^g}
	 	There exists some $Q'\in\NN(G,e)$ such that $Q^{\prime2}\subset Q$ and there is a $C^\infty$ leafwise homotopy of $\phi^{gh}$ to $\phi^h\phi^g$ for all $g,h\in Q'$. 
	 \end{claim}
	 
	Since $X$ is compact, there is $Q'\in\NN(G,e)$ such that $Q^{\prime2}\subset Q$ and
		\[
			\phi_j((\supp f_i\cap\supp f_j)\times Q')\subset U_i
		\]
	for all $i,j$. Then, for all $x\in\supp f_i\cap\supp f_j$ and $g\in Q'$, the points $\phi_i(x,g)$ and $\phi_j(x,g)$ are in the plaque $p_i^{-1}(\chi(p_i(x),g))$ of $(U_i,\xi_i)$ by Claim~\ref{cl: tilde p_i phi_i(x,g) = tilde p_i phi_j(x,g)}. Therefore $\phi(x,g)\in p_i^{-1}(\chi(p_i(x),g))$ according to Corollary~\ref{c: CC_bar lambda,bar x}. Applying again Claim~\ref{cl: tilde p_i phi_i(x,g) = tilde p_i phi_j(x,g)} in a similar way, we get that $\phi(\phi(x,g),h)$ is in the plaque of $(\widetilde U_i,\tilde\xi_i)$ over $\tilde\chi(\chi(p_i(x),g),h)=\tilde\chi(p_i(x),gh)$ for all $h\in Q'$. On the other hand, since $gh\in Q^{\prime2}\subset Q$, the same kind of argument shows that $\phi(x,gh)$ is in the plaque of $(\widetilde U_i,\tilde\xi_i)$ over $\tilde\chi(p_i(x),gh)$. Thus $\phi(\phi(x,g),h)$ and $\phi(x,gh)$ are in the same plaque of $(\widetilde U_i,\tilde\xi_i)$. Since these plaques are convex, we can use geodesic segments to construct a $C^\infty$ leafwise homotopy between the foliated maps $\phi^h\phi^g$ and $\phi^{gh}$ for all $g,h\in Q'$.
		 
	 \begin{claim}\label{cl: phi^g in Diffeo(X,FF)}
	 	There is some $Q''\in\NN(G,e)$ such that $Q''\subset Q'$ and $\phi^g\in\Diffeo(X,\FF)$ for all $g\in Q''$. 
	 \end{claim}
	
	For all $g\in Q'$, every restricted foliated map $\phi^g:U_i\to\widetilde U_i$ induces the open embedding $\tilde\chi^g:T_i\to\widetilde T_i$; i.e., $\{\,\phi^g\mid g\in Q'\,\}$ is a uniform family of transverse equivalences. Hence, since $\phi^e=\id_X$ and $g\mapsto\phi^g$ is continuous in the $C^\infty$ foliated topology, it follows from Proposition~\ref{p: Diffeo^r(FF,FF') cap MM is open in MM} that there is some $Q''\in\NN(G,e)$ such that $\phi^g\in\Diffeo(X,\FF)$ for all $g\in Q''$.
	
	From Claims~\ref{cl: phi^gh = phi^h phi^g} and~\ref{cl: phi^g in Diffeo(X,FF)}, and since $\phi^e=\id_X$, we get that $\phi:X\times Q\to X$ is a $C^\infty$ right transverse local action of $G$ on $X$. The induced right local action of $G$ on $T$ is equivalent to $\chi$ because every $\phi^g:U_i\to\widetilde U_i$ induces $\tilde\chi^g:T_i\to\widetilde T_i$.	
\end{proof}

Consider the following property that $(X,\FF,\phi)$ may have:
	\begin{equation}\label{FF(phi(x times P)) = X}
		\FF(\phi(\{x\}\times P))=X\quad\forall x\in X, \forall P\in\NN(G,e)\mid P\subset O.
	\end{equation}
	
\begin{lem}\label{l: FF(phi(x times P)) = X}
	Property~\eqref{FF(phi(x times P)) = X} is invariant by equivalences of right transverse local actions.
\end{lem}

\begin{proof}
	Elementary.
\end{proof}

\begin{lem}\label{l: FF(phi(x times P)) = X <=> HH(bar phi(u times Q)) = T}
	$(X,\FF,\phi)$ satisfies~\eqref{FF(phi(x times P)) = X} if and only if $(T,\HH,\bar\phi)$ satisfies~\eqref{HH(chi(u times P)) = T}.
\end{lem}

\begin{proof}
	Elementary.
\end{proof}

\subsection{Structural right transverse local action}\label{s: structural local transverse action}

Now, suppose that $X$ is a $C^\infty$ compact minimal $G$-foliated space. Fix any equivalence $\Psi$ of $\HH$ to the pseudogroup $\GG$ on $G$ generated by local left translations with respect to some finitely generated dense sublocal group $\Gamma\subset G$. The local multiplication $\mu:G\times G\rightarrowtail G$ is a right local action of $G$ on $G$ so that $\GG$ becomes locally equivariant. By Proposition~\ref{p: chi'}, there is a unique right local action $\chi:T\times  G\rightarrowtail T$, up to equivalences, such that $\HH$ and $\Psi$ become locally equivariant. According to Proposition~\ref{p: exists a right local transverse action}, there is a unique right local transverse action $\phi:X\times O\to X$ of $G$ on $X$ inducing $\chi$, up to equivalences, (whose equivalence class is) called the {\em structural right transverse local action\/}.

\section{$C^\infty$ $G$-foliated spaces are $C^\infty$ foliated homogeneous}
\label{s: C^infty G-fol sp => C^infty foliated homogeneous}

Suppose that $X$ is compact and $C^\infty$. Then the following result guarantees certain leafwise homogeneity.

\begin{prop}\label{p: leafwise homogeneous}
	Let $L$ be the leaf of $X$, let $D$ be a relatively compact regular domain without holonomy in $L$, and let $c:I\to D$ be any $C^\infty$ path. Then, for any open neighborhood $U$ of $c(I)$ in $X$, there is some $C^\infty$ leafwise diffeotopy $\phi:X\times I\to X$ supported in $U$ with $\phi(c(0),\cdot)=c$.
\end{prop}

\begin{proof}
	Let $E$ be a relatively compact open subset of $L$ such that $c(I)\subset E$ and $\ol E\subset D\cap U$. By the homogeneity of $L$, there is a diffeotopy $\psi:L\times I\to L$ supported in $E$ so that $\psi(\cdot,0)=\id_X$ and $\psi(c(0),\cdot)=c$. Let $\Sigma$ be a local transversal of $X$ through $x$. By the Reeb's stability theorem for $C^\infty$ foliated spaces \cite[Proposition~1.7]{AlvarezCandel2009}, there is a $C^\infty$ foliated embedding $h:D\times\Sigma\to X$ that can be identified with the identity on $D\times\{x\}\equiv D$ and $\{x\}\times\Sigma\equiv\Sigma$. Write $h^{-1}=(h',h''):\im h\to D\times\Sigma$. Take a compactly supported continuous function $f:\Sigma\to I$ with $h(\ol E\times\supp f)\subset U$ and $f(x)=1$. Then the statement is satisfied with the $C^\infty$ foliated diffeotopy $\phi:X\times I\to X$ defined by
		\[
			\phi(x,t)=
				\begin{cases}
					h(\psi(h'(x),fh''(x)),h''(x)) & \text{if $x\in\im h$}\\
					x & \text{otherwise}\;.\qed
				\end{cases}
		\]
\renewcommand{\qed}{}
\end{proof}

\begin{cor}\label{c: HH(bar phi(u times Q)) = T => X is C^infty foliated homogeneous}
	If there is a $C^\infty$ right transverse local action of $G$ on $X$ satisfying~\eqref{FF(phi(x times P)) = X}, then $X$ is $C^\infty$ foliated homogeneous.
\end{cor}

\begin{proof}
	Apply~\eqref{FF(phi(x times P)) = X} and Proposition~\ref{p: leafwise homogeneous}.
\end{proof}

\begin{proof}[Proof Theorem~\ref{mt: C^infty foliated homogeneous <=> G-fol sp}]
	By Theorem~\ref{mt: foliated homogeneous => G-foliated space}, it is enough to prove ``$\text{\eqref{i: G-foliated sp}}\Rightarrow\text{\eqref{i: C^infty foliated homogeneous}}$.'' With the notation of Section~\ref{s: structural local transverse action}, $(G,\GG,\mu)$ satisfies~\eqref{HH(chi(u times P)) = T} because 
		\[
			\mu((\Gamma\times\mu(\{g\}\times Q))\cap\dom\mu)=G
		\]
	for all $g\in G$ and $Q\in\NN(G,e)$ with $\{g\}\times Q\subset\dom\mu$. So $(T,\HH,\chi)$ also satisfies~\eqref{HH(chi(u times P)) = T} by Lemma~\ref{l: HH(chi(u times P)) = T}, and therefore $(X,\FF,\phi)$ satisfies~\eqref{FF(phi(x times P)) = X} by Lemma~\ref{l: FF(phi(x times P)) = X <=> HH(bar phi(u times Q)) = T}. Thus $X$ is $C^\infty$ foliated homogeneous by Corollary~\ref{c: HH(bar phi(u times Q)) = T => X is C^infty foliated homogeneous}
\end{proof}

\section{Examples and open problems}\label{s: examples}

\subsection{Molino's description of equicontinuous suspensions}\label{ss: suspensions}

Let $T$ be a compact space with a transitive left action of a compact topological group $G$, which is {\em quasi-analytic\/} in the sense that any $g\in G$ is the identity element $e\in G$ if it acts as the identity on some non-empty open set, and let $H\subset G$ be the isotropy group at some fixed point $u_0\in T$. Moreover let $\Gamma\subset G$ be a dense subgroup isomorphic to $\pi_1(M)/\pi_1(L)$ for some regular covering $L$ of some closed connected manifold $M$. Thus we have a right $\Gamma$-action on $L$ by covering transformations, and a left $\Gamma$-action on $T$ defined by the $G$-action. The induced diagonal $\Gamma$-action on $L\times T$, given by $(y,u)\cdot\gamma=(y\cdot\gamma,\gamma^{-1}\cdot u)$, is properly discontinuous and foliated, where $L\times T$ is foliated with leaves $L\times\{u\}$, for $u\in T$. The corresponding foliated quotient space, $L\times_\Gamma T$, is called the {\em suspension\/} of the $\Gamma$-action on $T$, and the quotient projection is a foliated covering map $L\times T\to L\times_\Gamma T$. The element in $L\times_\Gamma T$ defined by any $(y,u)\in L\times T$ will be denoted by $[y,u]$. Moreover the covering projection $\theta:L\to M$ induces a fiber bundle projection $\rho:L\times_\Gamma T\to M$, $\rho([y,u])=\theta(y)$, with typical fiber $T$; in particular, $L\times_\Gamma T$ is compact. Note that the fibers of $\rho$ are transverse to the leaves; i.e., $\rho:L\times_\Gamma T\to M$ is a flat bundle. Any flat bundle with compact total space is given by a suspension. 

Let us use the notation $X\equiv(X,\FF)$ for $L\times_\Gamma T$. Let $\VV=\{V_i,\zeta_i\}$ be an atlas of $M$, with $\zeta_i:V_i\to B_i$ for some contractible open subset $B_i\subset\R^n$. Thus the flat bundle $\rho:X\to M$ is trivial over every $V_i$; i.e., there are homeomorphisms $\psi_i:U_i:=\rho^{-1}(V_i)\to V_i\times T$ such that $\rho:U_i\to V_i$ corresponds to the first factor projection $V_i\times T\to V_i$ and the leaves of $\FF|_{U_i}$ correspond to the fibers of the second factor projection $V_i\times T\to T$. We get an induced foliated atlas $\UU=\{U_i,\xi_i\}$ of $X$, where $\xi_i=(\zeta_i\times\id_T)\psi_i:U_i\to B_i\times T'_i$ with $T'_i\equiv T$. Assuming obvious conditions on $\VV$, we get that $\UU$ is regular. Then $\UU$ induces a representative $\HH'$ of the holonomy pseudogroup of $X$ on $T'=\bigsqcup_iT'_i$. For any fixed index $i_0$, since $T'_{i_0}\equiv T$ meets all $\HH'$-orbits, by restricting $\HH'$ to $T'_{i_0}$, we get a pseudogroup $\HH$ on $T$ equivalent to $\HH'$, which is generated by the $\Gamma$-action on $T$. Thus $X$ is minimal, equicontinuous and strongly quasi-analytic (take $S=\Gamma$ to check the last two properties for $\HH$). Moreover $\ol\HH$ is generated by the $G$-action on $T$, and therefore $\ol\HH$ is also strongly quasi-analytic. So $X$ satisfies the conditions of Theorem~\ref{mt: Molino}.

Fix some $u_0\in T\equiv T'_{i_0}$, and consider the associated space $\widehat T'_0$ with the pseudogroup $\widehat\HH'_0$, and the associated representative of the Molino's description, $(G',\Gamma',H',\widehat X'_0\equiv(\widehat X'_0,\widehat\FF'_0),\hat\pi'_0)$, constructed like in the proof of Theorem~\ref{mt: Molino}. Then $\widehat T_0:=\widehat T'_{i_0,0}$ meets all $\widehat\HH'_0$-orbits, obtaining that $\widehat\HH'_0$ is equivalent to its restriction $\widehat\HH_0:=\widehat\HH'_0|_{\widehat T_0}$. Thus $\widehat T_0=\{\,\germ(g,u_0)\mid g\in G\,\}$ has the final topology induced by the map $G\to\widehat T_0$, $g\mapsto\germ(g,u_0)$. This map is a continuous bijection, and therefore it is a homeomorphism because $G$ is compact and $\widehat T_0$ is Hausdorff. So $\widehat T_0\equiv G$, $\widehat\HH$ is generated by the action of $G$ on itself by left translations, $G'$ is locally isomorphic to $G$, and $\hat\pi_0:\widehat T_0\equiv G\to T$ is the orbit map $g\mapsto g\cdot u_0$. The composition $\rho\hat\pi'_0:\widehat X'_0\to M$ is a fiber bundle with typical fiber $\widehat T_0\equiv G$, and $(\widehat X'_0,\rho\hat\pi'_0,\widehat\FF'_0)$ is also a flat bundle. Thus there is a foliated homeomorphism of $\widehat X'_0$ to $\widehat X_0\equiv(\widehat X_0,\widehat\FF_0):=L\times_\Gamma G$. Moreover
	\[
		H'\equiv H:=\{\,h\in G\mid h\cdot u_0=u_0\,\}\;,
	\]
the right $H'$-action on $\widehat X'_0$ corresponds to the right $H$-action on $\widehat X_0$ given by $[y,g]\cdot h=[y,gh]$, and the map $\hat\pi'_0:\widehat X'_0\to X$ corresponds to the map $\hat\pi_0:\widehat X_0\to X$ defined by $\hat\pi_0([y,g])=[y,g\cdot u_0]$, which is induced by the foliated map $\id_L\times\hat\pi_0:L\times G\to L\times T$. Thus $(G,\Gamma,H,\widehat X_0,\hat\pi_0)$ is another representative of the Molino's description, which will be used in the next examples.

If $M$ is $C^\infty$, its $C^\infty$ structure can be lifted to a $C^\infty$ structure on $L$, which in turn can be lifted to $L\times T$, which finally gives rise to a $C^\infty$ structure on $X$ so that the projection $\rho:X\to M$ is $C^\infty$ and $T\rho$ has isomorphic restrictions to the fibers. This can be similarly applied to $\widehat X_0$, obtaining the $C^\infty$ structure given by Proposition~\ref{p: C^infty Molino}. The same procedure can be applied to any Riemannian metric on $M$, obtaining induced Riemannian metrics on $X$ and $\widehat X_0$ so that the projections $\rho:X\to M$ and $\hat\pi_0:\widehat X_0\to X$ have locally isometric restrictions to the leaves.

The following result is well known. A proof is included for completeness.

\begin{prop}\label{p: the Gamma-action on T has no fixed points}
	The following properties are equivalent:
		\begin{enumerate}
		
			\item\label{i: the Gamma-action on T has no fixed points} The $\Gamma$-action on $T$ has no fixed points.
			
			\item\label{i: Gamma cap gHg^-1 = e} $\Gamma\cap gHg^{-1}=\{e\}$ for all $g\in G$. 
			
			\item\label{i: homeomorphic leaves} The canonical foliated projection $L\times T\to X$ restricts to homeomorphisms between the leaves.
			
		\end{enumerate}
\end{prop}

\begin{proof}
	Let us prove ``$\text{\eqref{i: the Gamma-action on T has no fixed points}}\Leftrightarrow\text{\eqref{i: Gamma cap gHg^-1 = e}}$''. Given any $\gamma\in\Gamma$ and $u\in T$, take some $g\in G$ such that $u=g\cdot u_0$. Then
	\begin{align*}
		\gamma u=u&\Leftrightarrow\gamma g\cdot u_0=g\cdot u_0\Leftrightarrow g^{-1}\gamma g\cdot u_0=u_0\\
		&\Leftrightarrow g^{-1}\gamma g\in H\Leftrightarrow\gamma\in\Gamma\cap gHg^{-1}=\{e\}\Leftrightarrow\gamma=e\;.
	\end{align*}

	Let us prove ``$\text{\eqref{i: the Gamma-action on T has no fixed points}}\Leftrightarrow\text{\eqref{i: homeomorphic leaves}}$''. For all $y,y'\in L$ and $u\in T$, we have $[y,u]=[y',u]$ if and only if there is some $\gamma\in\Gamma$ such that $(y',u)=(y\cdot\gamma,\gamma^{-1}\cdot u)$, which means $\gamma=e$ and $y'=y$.
\end{proof}

When the conditions of Proposition~\ref{p: the Gamma-action on T has no fixed points} are satisfied, $X$ is strongly locally free (in particular, it has no holonomy), and all leaves are homeomorphic to $L$. If moreover $M$ is $C^\infty$/Riemannian, then $L\times T\to X$ restricts to diffeomorphisms/isometries between the leaves, obtaining that all leaves are diffeomorphic/isometric to $L$.

\subsection{The map $\hat\pi_0:\widehat X_0\to X$ may not be a principal bundle}\label{ss: hat pi_0 is not a principal bundle}

Consider the canonical inclusion $\SO(2)\subset\SO(3)$, and the canonical transitive analytic action of $\SO(3)$ on the sphere $S^2\equiv\SO(3)/\SO(2)$. We get an induced transitive quasi-analytic left action of the compact topological group $G:=\SO(3)^\N$ on the compact space $T:=(S^2)^\N$. Fix $u_0\in S^2$ whose isotropy group is $\SO(2)$, and let $\bar u_0=(u_0,u_0,\dots)\in T$. The orbit map $\SO(3)\to S^2$, $g\mapsto g\cdot u_0$, is a non-trivial principal $\SO(2)$-bundle, and therefore it has no global sections. Then, using the arguments of the first and second examples of \cite[Section~1]{Ramsay1991}, it easily follows that the orbit map $G\to T$, $(g_i)\mapsto(g_i)\cdot\bar u_0=(g_i\cdot u_0)$, has no local sections. Since $G$ is second countable, connected, compact and non-abelian, it contains a dense subgroup $\Gamma$ isomorphic to the fundamental group of the closed oriented surface $\Sigma_2$ of genus $2$ \cite[Corollary~8.3]{BreuillardGelanderSoutoStorm2006}. Let $L$ be the universal covering of $\Sigma_2$, which is diffeomorphic to the plane. Consider the corresponding suspension foliated space, $X=L\times_\Gamma T$, which satisfies the conditions of Theorem~\ref{mt: Molino}, and the corresponding Molino's description $(G,\Gamma,H,\widehat X_0,\hat\pi_0)$ constructed in Section~\ref{ss: suspensions}, where $\widehat X_0=L\times_\Gamma G$, $H=\SO(2)^\N$, the right $H$-action on $\widehat X_0$ is given by $[y,g]\cdot h=[y,gh]$, and the map $\hat\pi_0:\widehat X_0\to X$ is defined by $\hat\pi_0([y,g])=[y,g\cdot u_0]$.

\begin{prop}\label{p: hat pi_0 is not a principal bundle}
	The map $\hat\pi_0:\widehat X_0\to X$ has no local sections, and therefore it cannot be a principal $H$-bundle. 
\end{prop}

\begin{proof}
	Since $\hat\pi_0:\widehat X_0\to X$ is induced by $\id_L\times\hat\pi_0:L\times G\to L\times T$, any local section of $\hat\pi_0$ with small enough domain defines a local section of $\hat\pi_0:G\to T$. But this map has no local sections.
\end{proof}

\subsection{Foliated homogeneity may not be told by the leaves}
\label{ss: foliated homogeneity may not be told by the leaves}

\begin{prop}\label{p: foliated homogeneous => the leaves are homeomorphic one another}
If $X$ is foliated homogeneous, then it is without holonomy, and all of its leaves are homeomorphic one another. If moreover $X$ is $C^\infty$ (respectively, compact and Riemannian), then all of its leaves are diffeomorphic (respectively, quasi-isometrically diffeomorphic) to each other.
\end{prop}

\begin{proof}
	Elementary, using that there always exist leaves without holonomy in the first assertion, and using that the differentiable quasi-isometry class of the leaves is independent of the choice of the Riemannian metric on $X$ in the last assertion (see e.g.\ \cite[Proposition~10.5]{AlvarezCandel2018}). 
\end{proof}

Let us exhibit an example where the reciprocal of Proposition~\ref{p: foliated homogeneous => the leaves are homeomorphic one another} does not hold. To begin with, let $G_1$ and $G_2$ be second countable, connected compact topological groups, and let $G=G_1\times G_2$. Assume that $G_1$ is non-abelian. Let us use the notation $g=(g_1,g_2)$ for the elements of $G$; in particular, we use $e=(e_1,e_2)$ for the identity element.

\begin{prop}\label{p: PP}
	There exists a subset $\PP\subset G\times G$, which is both residual and of full Haar measure, such that, for all $(g,h)\in\PP$, the subgroup $\langle g,h\rangle$ is dense in $G$ and freely generated by $g$ and $h$, and $\langle g,h\rangle\cap(\{e_1\}\times G_2)=\{e\}$.
\end{prop}

\begin{proof}
	By \cite[Proposition~8.2]{BreuillardGelanderSoutoStorm2006}, there are subsets, $\OO\subset G\times G$ and $\OO_1\subset G_1\times G_1$, which are residual and of full Haar measure, such that, for all $(g,h)\in\OO$ and $(a,b)\in\OO_1$, the subgroup $\langle g,h\rangle$ (respectively, $\langle a,b\rangle$) is dense in $G$ (respectively, $G_1$) and freely generated by $g$ and $h$ (respectively, $a$ and $b$). Then the statement is satisfied with
		\[
			\PP=\OO\cap\{\,(g,h)\in G\times G\mid(g_1,h_1)\in\OO_1\,\}\;.\qed
		\]
\renewcommand{\qed}{}
\end{proof}

Take $G_2=\SO(3)$, and consider $\SO(2)\subset\SO(3)$ and $S^2\equiv\SO(3)/\SO(2)$ like in Section~\ref{ss: hat pi_0 is not a principal bundle}. By Proposition~\ref{p: PP}, $G$ has a dense subgroup $\Gamma$ freely generated by two elements such that $\Gamma\cap(\{e_1\}\times\SO(3))=\{e\}$. Hence the first factor projection $G_1\times\SO(3)\to G_1$ restricts to an injection $\Gamma\to G_1$, and $\Gamma$ does not meet any conjugate of $\{e_1\}\times\SO(2)$ in $G$ (all of them are contained in $\{e_1\}\times\SO(3)$). Consider the canonical left action of $G$ and $\Gamma$ on $T:=G_1\times S^2\equiv G/(\{e_1\}\times\SO(2))$. There is a regular covering $L$ of the closed oriented surface of genus two, $\Sigma_2$, whose group of covering transformations is isomorphic to $\Gamma$. Consider the corresponding suspension foliated space, $X=L\times_\Gamma T$, which satisfies the conditions of Theorem~\ref{mt: Molino}, and the corresponding Molino's description $(G,\Gamma,H,\widehat X_0,\hat\pi_0)$ constructed in Section~\ref{ss: suspensions}, where $\widehat X_0=L\times_\Gamma G$, $H=\SO(2)$, the right $H$-action on $\widehat X_0$ is given by $[y,g]\cdot h=[y,gh]$, and the map $\hat\pi_0:\widehat X_0\to X$ is defined by $\hat\pi_0([y,g])=[y,g\cdot u_0]$. We can equip $\Sigma_2$ with $C^\infty$ and Riemannian structures, and consider the induced $C^\infty$ and Riemannian structures on $X$ and $\widehat X_0$. 

Since $H\ne\{e\}$, $X$ is not foliated homogeneous by Theorem~\ref{mt: C^infty foliated homogeneous <=> G-fol sp} (or Theorem~\ref{mt: foliated homogeneous => G-foliated space} and Proposition~\ref{p: G-foliated space <=> H = e}). However this cannot be seen by comparing any pair of leaves since all of them are isometric to $L$, and $X$ has no holonomy by ``$\text{\eqref{i: Gamma cap gHg^-1 = e}}\Leftrightarrow\text{\eqref{i: homeomorphic leaves}}$'' in Proposition~\ref{p: the Gamma-action on T has no fixed points}.

This argument cannot produce matchbox manifolds because Proposition~\ref{p: PP} requires $G$ to be connected to apply \cite[Proposition~8.2]{BreuillardGelanderSoutoStorm2006}. Examples with totally disconnected local transversals are given in  \cite[Theorem~35]{FokkinkOversteegen2002} and \cite[Theorem~10.7]{DyerHurderLukina2017}.

\subsection{Inverse limits of minimal Lie foliations}\label{ss: inverse limits of Lie foliations}

This example was suggested by S.~Hurder. Let $(X,\GG)$ be the McCord solenoid defined as the projective limit of a tower of non-trivial regular coverings between closed connected manifolds,
	\[
		\begin{CD}
			\cdots \to M_k @>{\phi_k}>> M_{k-1}\to \cdots \to M_0\;.
		\end{CD}
	\]
Let $\Gamma_k=\pi_1(M_k)$, and consider the induced tower of homomorphisms between finite groups,
	\[
		\cdots\to\Gamma_0/\Gamma_k\to\Gamma_0/\Gamma_{k-1}\to\cdots\to\Gamma_0/\Gamma_1\;,
	\]
whose inverse limit $K$ contains a canonical dense copy of $\Gamma_0$. Then $(X,\GG)$ can be also described as the suspension foliated space $\widetilde M_0\times_{\Gamma_0}K$, where $\widetilde M_0$ is the universal covering of $M_0$. We get induced maps $\psi_k:X\to M_k$, whose restrictions to the leaves are covering maps. Suppose that $M_0$ is equipped with a minimal Lie $G_0$-foliation $\FF_0$, for some simply connected Lie group $G_0$. Then every $M_k$ can be endowed with the minimal Lie $G_0$-foliation $\FF_k:=(\phi_1\cdots\phi_k)^*\FF_0$. On every $\GG$-leaf $M$, consider the pull-back of $\FF_0$ by $\psi_0:M\to M_0$. These foliations on all leaves of $\GG$ can be combined to form a foliated structure $\FF$ on $X$, which is a ``Lie $G_0$-subfoliated structure'' of $\GG$ in an obvious sense. We can write $\FF=\psi_0^*\FF_0$, which equals $\psi_k^*\FF_k$ for all $k$. Extending the notation of suspensions, we can also write $(X,\FF)=(\widetilde M_0,\widetilde\FF_0)\times_{\Gamma_0}K$, where $\widetilde\FF_0$ is the lift of $\FF_0$. It easily follows that $(X,\FF)$ is a minimal $G$-foliated space for $G=G_0\times K$.

\subsection{Open problems}\label{ss: open problems}

\subsubsection{Strong quasi-analyticity of $\ol\HH$}

This problem was proposed in \cite{AlvarezMoreira2016}. It is really unknown to the authors if the strong quasi-analyticity of $\ol\HH$ is needed in Theorem~\ref{mt: Molino}. More precisely, assuming that $\HH$ is a minimal compactly generated equicontinuous strongly quasi-analytic pseudogroup, is $\ol\HH$ strongly quasi-analytic? If minimality is not assumed, then counterexamples can be easily given. But the minimal case seems to be an interesting open problem. Among the wild matchbox solenoids of \cite{HurderLukina-Wild} there might be counterexamples. 

\subsubsection{Functoriality, universality and uniqueness of the Molino's description}

It would be desirable to have a uniqueness of the Molino's description stronger than Proposition~\ref{p: equivalent Molino}, stating that not only the structures $(G,\Gamma,H,\widehat X_0,\hat\pi_0)$ constructed in the proof of Theorem~\ref{mt: Molino}, but also all possible structures $(G,\Gamma,H,\widehat X_0,\hat\pi_0)$ satisfying the conditions of its statement are equivalent. This would follow by showing a universality property, which in turn would follow by exhibiting its functoriality with respect to some kind of foliated maps. Since the definition of $\widehat X_0$ uses germs of maps in $\ol\HH$, the functoriality of Molino's description could be achieved by showing that foliated maps between equicontinuous foliated spaces induce morphisms between the closures of their holonomy pseudogroups. This would be an extension of the case of Riemannian foliations, solved in \cite{AlvarezMasa2008,AlvarezMasa2006}. Such functoriality, universality and uniqueness of the Molino's description is not even proved in the Riemannian foliation case. A direct consequence would be that $H$ is finite if and only if $X$ is a {\em virtually foliated homogeneous\/} foliated space (a finite fold covering of $X$ is foliated homogeneous as a foliated space). When $X$ is an equicontinuous matchbox manifold, Dyer, Hurder and Lukina have shown that, if $H$ is finite, then $X$ is virtually homogeneous (a finite fold covering of $X$ is homogeneous) \cite[Theorem 1.13 and Corollary~1.14]{DyerHurderLukina2016}.

\subsubsection{How large is the class of inverse limits of minimal Lie foliations?}

Since any metrizable locally compact local group of finite topological dimension is locally isomorphic to the direct product of a Lie group and a compact zero-dimensional topological group \cite[Theorem~107]{Jacoby1957}, it was asked by S.~Hurder whether any compact minimal foliated homogeneous foliated space of finite ``topological codimension'' can be realized as inverse limit of minimal Lie foliations, like in Section~\ref{ss: inverse limits of Lie foliations}. This would generalize the results of \cite{ClarkHurder2013} (see also \cite{AlcaldeLozanoMacho2011}), where an affirmative answer is given for homogeneous matchbox manifolds (the case of codimension zero). If this is true, using also the Molino's description, it could be possible to prove that any equicontinuous foliated space satisfying the conditions of Theorem~\ref{mt: Molino} is an inverse limit of Riemannian foliations.

\subsubsection{Molino's descriptions without assuming strong quasi-analyticity}

This problem arises from the Molino spaces constructed by Dyer, Hurder and Lukina in \cite{DyerHurderLukina2017} for equicontinuous matchbox manifolds, where strong quasi-analyticity is not needed. Their Molino spaces are also foliated homogeneous, and their leaves cover the leaves of the original matchbox, but they may not be unique. Thus the following question is natural: does there exist this kind of Molino spaces for arbitrary compact minimal equicontinuous foliated spaces?



\providecommand{\bysame}{\leavevmode\hbox to3em{\hrulefill}\thinspace}
\providecommand{\MR}{\relax\ifhmode\unskip\space\fi MR }
\providecommand{\MRhref}[2]{%
  \href{http://www.ams.org/mathscinet-getitem?mr=#1}{#2}
}
\providecommand{\href}[2]{#2}

\end{document}